\documentclass{numapde-preprint}

\usepackage{numapde-semantic}
\usepackage{numapde-style}
\usepackage{numapde-local}

\hypersetup{
	pdftitle={First-Order Methods for Optimal Experimental Design Problems with Bound Constraints},
	pdfauthor={Roland Herzog, Eric Legler},
	pdfkeywords={first-order methods, optimal experimental design, bound constraints, FISTA, proximal extrapolated gradient method, simplicial decomposition}
}

\title{First-Order Methods for Optimal Experimental Design Problems with Bound Constraints}
\subtitle{}
\shorttitle{First-Order Methods for Bound Constrained OED Problems}

\author{Roland Herzog\thanks{Technische Universität Chemnitz, Faculty of Mathematics, 09107 Chemnitz, Germany (\email{roland.herzog@mathematik.tu-chemnitz.de}, \url{https://www.tu-chemnitz.de/mathematik/part_dgl/people/herzog}, \orcid{0000-0003-2164-6575}, \email{eric.legler@mathematik.tu-chemnitz.de}, \url{https://www.tu-chemnitz.de/mathematik/part_dgl/people/legler}, \orcid{0000-0003-1604-5601}).}
\and
Eric Legler\footnotemark[1]}
\shortauthor{R. Herzog and E. Legler}

\dedication{}

\IfFileExists{numapde-uncolorizeHyperrefIfChanges.sty}{\RequirePackage{numapde-uncolorizeHyperrefIfChanges}}{}

\begin{document}
\maketitle

% Insert abstract
\begin{abstract}
We consider a class of convex optimization problems over the simplex of probability measures.
Our framework comprises optimal experimental design (OED) problems, in which the measure over the design space indicates which experiments are being selected.
Due to the presence of additional bound constraints, the measure possesses a Lebesgue density and the problem can be cast as an optimization problem over the space of essentially bounded functions.
For this class of problems, we consider two first-order methods including FISTA and a proximal extrapolated gradient method, along with suitable stopping criteria.
Finally, acceleration strategies targeting the dimension of the subproblems in each iteration are discussed.
Numerical experiments accompany the analysis throughout the paper.\end{abstract}

% Insert keywords

% Insert Mathematics Subject Classification (MSC2010)

% Insert document body
%------------------------------------------------------------------
\section{Introduction}
\label{sec:Introduction}
%------------------------------------------------------------------

In this paper we consider problems of the following form:
\begin{equation}
	\label{eq:OED_unregularized}
	\begin{aligned}
		\text{Minimize} \quad & F(\Lambda w), \quad w \in C(X)^* \\
		\text{s.t.} \quad & w \in \Delta_C \cap \bounds 
		.
	\end{aligned}
\end{equation}
Here $X$ is a compact subset of $\R^d$, and $C(X)^*$ is the dual of the space of continuous functions on $X$, which can be represented by signed, finite, regular Borel measures on $X$; see for instance \cite[Prop.~7.16]{Folland1984} or \cite[Thm.~6.19]{Rudin1987}.
The set $\Delta \subset C(X)^*$ is the standard simplex of probability measures, 
\begin{equation*}
	\Delta = \setDef{w \in C(X)^*}{w \ge 0 \text{ and } w(X) = 1}
\end{equation*}
and $\Delta_C = C \delta$ for some $C>0$.
Moreover, $\Lambda \in \LL(C(X)^*,\S^n)$ is a bounded linear operator where $\S^n$ is the space of symmetric $n$-by-$n$ matrices.
Finally, $F: \S^n \to \overline \R \coloneqq \R \cup \{+\infty\}$ is a proper, convex, lower semi-continuous function.

In the absence of $\bounds$, problems of type \eqref{eq:OED_unregularized} arise in the formulation of continuous sensor placement or, more generally, optimal experimental design (OED) problems for parameter identification; see \cite[Ch.~3]{Ucinski2005:1}, \cite[Ch.~5]{PronzatoPazman2013}.
In this setting, $X$ is termed the design space and $w$ is the design measure, which describes the weighted composition of an ensemble of measurements (a complete experiment) from elementary measurements (experiments).
In this context, $n$ is the dimension of the space of unknown parameters, and $\Upsilon(x) \in \S^n$ denotes the positive-semidefinite Fisher information of the elementary experiment associated with the design point $x$.
We assume that $\Upsilon$ depends continuously on $x$, i.e., $\Upsilon \in C(X;\S^n)$ and that
\begin{equation}
	\label{eq:synthesis_operator}
	\Lambda w \coloneqq \int_X \Upsilon(x) \d w(x)
\end{equation}
denotes the total Fisher information matrix (FIM) of the experiment with weight $w \in C(X)^*$.

In this paper we focus on OED problems of type \eqref{eq:OED_unregularized} in the presence of additional bound constraints.
Such problems have been discussed before, e.g., in \cite{Fedorov1989,CookFedorov1995,Pronzato2004,MolchanovZuyev2004,PatanUcinski2017} and \cite[Chapter~4.3]{FedorovHackl1997:1}.
Specifically, we consider
\begin{equation}
	\label{eq:OED_pointwise_constraints}
	\bounds \coloneqq \setDef{w \in C(X)^*}{0 \le w(A) \le \abs{A} \text{ for all Borel sets } A \subset X}
	,
\end{equation}
where $\abs{A}$ denotes the Lebesgue measure of $A$.
The constraints in \eqref{eq:OED_pointwise_constraints} can represent, among other things, an upper bound on the density of sensors placed in any given subset~$A$ of the observation domain~$X$.
Naturally, we need to assume $0 \le C \le \abs{X}$ in order for \eqref{eq:OED_unregularized} to be feasible.
Notice that \eqref{eq:OED_unregularized}--\eqref{eq:criteria} is a convex problem.

The presence of $\bounds$ changes problem \eqref{eq:OED_unregularized} quite a lot.
First of all, \eqref{eq:OED_pointwise_constraints} implies that the Borel measure~$w$ is absolutely continuous w.r.t.\ the Lebesgue measure on~$X$.
This implies that $w$ possesses a density (Radon--Nikodym derivative) in $L^1(X)$ w.r.t.\ the Lebesgue measure.
In a slight abuse of notation, we denote this density also by $w$.
Moreover, \eqref{eq:OED_pointwise_constraints} implies that $0 \le w \le 1$ holds a.e.\ in $X$ and thus in fact $w$ belongs to $L^\infty(X)$.

Second, the structure of optimal solutions of \eqref{eq:OED_unregularized} depends on whether or not $\bounds$ is present.
In the absence of $\bounds$, the set of optimal measures contains an element which has sparse support and is, in fact, a non-negative linear combination of at most $n \, (n+1)/2$ Dirac measures.
This is a consequence of Carath{\'e}odory's theorem; see for instance \cite[Thm.~3.1]{Ucinski2005:1}.
With $\bounds$ present, the set of optimal measures contains a representer which is of bang-bang type, i.e., $w(x) \in \{0,1\}$ holds a.e.\ in $X$; see \cite[Thm.~4.3.1]{FedorovHackl1997:1}.

A typical class of design criteria for OED problems is given by 
\begin{equation}
	\label{eq:criteria}
	F_q(\FIM) =
	\begin{cases}
		\paren[big](){\frac{1}{n} \trace \FIM^{-q}}^{\frac{1}{q}}
		&
		\text{for } q \in (0, \infty),
		\\
		\frac{1}{n} \ln \det \FIM^{-1}
		&
		\text{for } q = 0
	\end{cases}
\end{equation}
for positive definite matrices $\FIM \in \S^n$ (denoted by $\FIM \succ 0$); see for instance \cite[eq.(4.11)]{Kiefer1974} or \cite[eq.(2.17)]{Ucinski2005:1}. 
For the sake of simplicity, we omitted a multiplicative constant.
Important special cases are the A-criterion ($q = 1$), the logarithmic D-criterion ($q = 0$) and the E-criterion $F_\infty(I) = \max \operatorname{eig}(I^{-1})$ obtained in the limit $q \to \infty$.
When $\FIM \in \S^n$ is not positive definite, we set $F_q(\FIM) \coloneqq +\infty$.
Notice that all criteria satisfy the assumptions of being proper and convex functionals on $\S^n$.
Moreover, all except $F_\infty$ are of class $C^\infty$ on their domain.
Notice that for $I \succ 0$ we have $F_q(I) \to (\det I^{-1})^{1/n}$ as $q \searrow 0$, see \cite[eq.(4.11)]{Kiefer1974}, and thus the criterion $F_0$ is not the limit of $F_q$ but its logarithm.
However it is customary to use the equivalent criterion given in \eqref{eq:criteria} in case $q = 0$.

There are several classes of numerical approaches to solving \eqref{eq:OED_unregularized}--\eqref{eq:criteria} in the absence of $\bounds$.
An early, so-called additive approach has been presented by \cite[Ch.~2.5]{Fedorov1972,Wynn1970}, which proceeds by taking convex combinations of the current iterate and a particular vertex of the probability simplex~$\Delta_C$.
This idea has been further improved by \cite{NeitzelPieperVexlerWalter2018} in order to ensure the sparsity of the solution within each iteration.
By contrast, so-called multiplicative algorithms have been introduced for discretized versions of the problem, which modify all weights simultaneously, see  \cite{SilveyTitteringtonTorsney1978:1,Torsney2009:1,Yu2010:1}.
Alternative approaches are obtained by reformulating the problem into a second order cone or semidefinite program, see for instance \cite{SagnolHarman2015,AlizadehGoldfarb2003}.

We also mention problems which feature binary constraints $\bounds' = \setDef{w \in C(X)^*}{w(\{x\}) \in \{0,1\} \text{ for all } x \in X}$.
Hence in contrast to $\bounds$, only Dirac measures are allowed.
Examples include so-called exact sensor placement problems, where one seeks to identify the most useful positions for a given number of sensors.
Typical algorithms for a discrete versions of this problem are either sequential or of exchange type.
Sequential methods insert one sensor at a time, yielding suboptimal solutions; see for instance \cite{Papadimitriou2004:1,Willcox2006:1,HerzogRiedel2013:1}.
Exchange algorithms revisit previously placed sensors and replace a subset of them by unused ones, often heuristically.
For more details see \cite[Ch.~12]{AtkinsonDonevTobias2007}, \cite[Ch~4.3]{FedorovHackl1997:1} or \cite[Ch.~7]{Ucinski2005:1}. 

An exact approach to discretized exact sensor placement problems of branch-and-bound type, which takes into account the combinatorial nature, is discussed in \cite{UcinskiPatan2007}.
It requires the solution of subproblems which are discretized variants of \eqref{eq:OED_unregularized}--\eqref{eq:criteria}.
In order to solve these subproblems, the authors make use of a simplicial decomposition algorithm, where the inner, so-called restricted master problems, are solved using a multiplicative algorithm.
Recently, \cite{EtlingHerzogSiebenborn2018:1} proposed to replace the latter by FISTA (\cite{BeckTeboulle2009}), which significantly reduced the computational time.

The goal of this paper is to formulate, discuss and numerically compare first-order methods for the solution of \eqref{eq:OED_unregularized}--\eqref{eq:criteria} and its discrete variants.
First-order methods are popular for a variety of smooth and non-smooth optimization problems; see for instance \cite{EsserZhangChan2010,BurgerSawatzkySteidl2016,Beck2017}.
To the best of our knowledge, first-order methods (with or without simplicial decomposition) have not been considered for the solution of optimal experimental design problems with bound constraints~$\bounds$.

For the sake of completeness, we embed \eqref{eq:OED_unregularized}--\eqref{eq:criteria} into a family of regularized problems 
\begin{equation}
	\label{eq:OED_regularized}
	\begin{aligned}
		\text{Minimize} \quad & F_q(\Lambda w) + \frac{\alpha}{2} \norm{w}_{L^2(X)}^2, \quad w \in L^2(X) \\
		\text{s.t.} \quad & w \in \Delta_C \cap \bounds.
	\end{aligned}
\end{equation}
with parameter $\alpha \ge 0$.
Recall that the fact that the design measure~$w$ is sought in $L^2(X)$ is not a result of the regularization term but rather a consequence of the bound constraints in $\bounds$.
The case $\alpha > 0$ is interesting since the solution $w$ is unique, which is in general not the case for $\alpha = 0$.
Finally, we are going to shows that a choice of $\alpha > 0$ does not jeopardize the sparsity of the optimal design measure, which is an important consideration for the practical realization.

Taking into account that $w$ has a Lebesgue density, the synthesis operator \eqref{eq:synthesis_operator} can be extended to $\Lambda \in \LL(L^2(X),\S^n)$ and it can be written as
\begin{equation}
	\label{eq:synthesis_operator_with_density}
	\Lambda w \coloneqq \int_X \Upsilon(x) \, w(x) \dx
	.
\end{equation}
Here it is sufficient to assume $\Upsilon = L^2(X;\S^n)$.
For later reference, we note that the Hilbert space adjoint $\Lambda^* \in \LL(\S^n,L^2(X))$ of $\Lambda$ is given by
\begin{equation}
	\label{eq:synthesis_operator_with_density_Hilbert_adjoint}
	(\Lambda^* P)(x) = P \dprod \Upsilon(x) 
	,
\end{equation}
where $A \dprod B = \trace(A^\transp B)$ for matrices of equal size.
Moreover, we note that the scaled probability simplex in $L^2(X)$ can be written as 
\begin{equation}
	\label{eq:simplex_in_L2}
	\Delta_C 
	\coloneqq 
	\setDef[auto]{w \in L^2(X)}{w \ge 0 \text{ a.e.\ in $X$ and } \int_X w \dx = C} 
\end{equation}
and the bound constraints are
\begin{equation}
	\label{eq:bounded_in_L2}
	\bounds \coloneqq \setDef[auto]{w \in L^2(X)}{0 \le w \le 1 \text{ a.e.\ in $X$}}
	.
\end{equation}
Throughout, we assume that \eqref{eq:OED_regularized}--\eqref{eq:bounded_in_L2} is feasible and that 
\begin{equation}
	\label{eq:assumption_C}
	0 < C < \abs{X}
\end{equation}
holds, i.e., the solution is neither $w \equiv 0$ nor $w \equiv 1$.
Moreover, we assume that there exists $w \in \Delta_C \cap \bounds$ such that $F_q(\Lambda w) < \infty$ holds, i.e., there exists at least one feasible experiment which yields a positive definite Fisher information matrix.

\begin{proposition}[Existence of a solution]
	\label{prop:existence}
	Problem \eqref{eq:OED_regularized}--\eqref{eq:bounded_in_L2} is solvable.
	In case $\alpha > 0$, the solution is unique.
\end{proposition}
\begin{proof}
	The set of attainable Fisher information matrices $\Lambda(\Delta_C \cap \bounds)$ is bounded since $\Delta_C \cap \bounds$ is bounded in $L^2(X)$ and $\Lambda \in\LL(L^2(X),\S^n)$ is continuous.
	It is easy to see that the objective in \eqref{eq:OED_regularized} is bounded below on $\Delta_C \cap \bounds$ since the set of .
	Consequently, a direct proof based on a minimizing sequence is applicable.
	The uniqueness of $w$ in case $\alpha > 0$ follows from the uniform convexity of the objective in this case.
\end{proof}

The structure of the paper is as follows.
In \cref{sec:optimality_conditions} we investigate necessary and sufficient optimality conditions for the family of convex problems \eqref{eq:OED_regularized}--\eqref{eq:bounded_in_L2}.
A discrete version of the problem is considered in \cref{sec:Discretization_DG0}.
In \cref{sec:numerical_results} we compare the performance of two first-order methods using two numerical experiments.
Specifically, we discuss FISTA (\cite{BeckTeboulle2009}) and the proximal extrapolated gradient method from \cite{Malitsky2017}.
In \cref{sec:simplicial_decomposition} we discuss the embedding of the methods considered into a simplicial decomposition acceleration, which features subproblems of reduced dimension and significantly improves the performance.

A result which may be of independent interest is an efficient algorithm for the orthogonal projection (w.r.t.\ an inner product given by $\diag(e_i)$) onto the restricted simplex 
\begin{equation*}
	\setDef[auto]{w \in \R^m}{0 \le w_i \le 1 \text{ for all } i = 1, \ldots, m \text{ and } \sum_{i=1}^m e_i \, w_i = C}
	.
\end{equation*}
The algorithm is given as \cref{alg:projection_feasible_set} and its correctness is proved in \cref{theorem:correctness_of_alg:projection_feasible_set}.

%------------------------------------------------------------------
\section{Optimality Conditions}
\label{sec:optimality_conditions}
%------------------------------------------------------------------

In this section we investigate necessary and sufficient optimality conditions for problem \eqref{eq:OED_regularized}--\eqref{eq:bounded_in_L2}, which can be equivalently cast in the following typical form for convex optimization problems:
\begin{equation}
	\label{eq:decomposition_into_F_G_H_primal}
	\text{Minimize} \quad F_q(\Lambda w) + G(w) + H_\alpha(w), \quad w \in L^2(X)
	.
\end{equation}
Here we set 
\begin{equation}
	\label{eq:definitions_of_G_H_primal}
	G(w) \coloneqq I_{\Delta_C \cap \bounds}(w)
	\quad \text{and} \quad
	H_\alpha(w) \coloneqq \frac{\alpha}{2} \norm{w}_{L^2(X)}^2
\end{equation}
and $I_{\Delta_C \cap \bounds}$ denotes the indicator function (in the sense of convex analysis) of the set $\Delta_C \cap \bounds \subset L^2(X)$.
We recall that we assume \eqref{eq:assumption_C} throughout.

Both $L^2(X)$ and $\S^n$, endowed with their inner products $\int_X w \, v \dx$ and $P \dprod Q \coloneqq \trace(P^\transp Q)$, respectively, are Hilbert spaces and we identify them with their duals via the respective Riesz isomorphisms.
Notice that $P \dprod P = \norm{P}_F^2$ holds, where $\norm{\cdot}_F$ is the Frobenius norm of matrices.

Before we present optimality conditions for \eqref{eq:OED_regularized}--\eqref{eq:bounded_in_L2}, we relate the $L^2(X)$-orthogonal projection $\proj{\Delta_C \cap \bounds}$ to the projection onto $\bounds$.
The latter is simply 
\begin{equation*}
	\proj{\bounds}{f} 
	= 
	\max \{ \min \{f,1 \}, 0 \}.
\end{equation*}
The following result shows that projection onto $\Delta_C \cap \bounds$ is obtained by projection onto $\bounds$ after an appropriate constant shift of the argument.

\begin{lemma}
	\label{lemma:projection}
	Let $f \in L^2(X)$.
	\begin{enumerate}[label=$(\alph*)$]
		\item 
			\label{item:projection_1}
			Suppose that $v = \proj{\Delta_C \cap \bounds}{f}$ holds.
			Then there exists $\zeta \in \R$ such that 
			\begin{equation}
				\label{eq:proj_Delta_bounds}
				v = \proj{\bounds}{f - \zeta} 
				\quad \text{and} \quad
				\int_X v \d x = C
			\end{equation}
			holds.

		\item 
			\label{item:projection_2}
			Conversely, suppose that $(v,\zeta) \in L^2(X) \times \R$ satisfies \eqref{eq:proj_Delta_bounds}, then $v = \proj{\Delta_C \cap \bounds}{f}$.
	\end{enumerate}
\end{lemma}
\begin{proof}
	We utilize that $v = \proj{\Delta_C \cap \bounds}{f}$ is characterized by the necessary and sufficient optimality condition
	\begin{equation}
		\label{eq:projection_1}
		\tag{$*$}
		\inner{v - f}{w - v}_{L^2(X)} \ge 0 \quad \text{for all } w \in \Delta_C \cap \bounds
		,
	\end{equation}
	see for instance \cite[Ch.~II, eq.(2.17)]{EkelandTemam1999}.
	\Cref{item:projection_1}:
	Let us define 
	\begin{equation*}
		\R \ni \zeta \mapsto a(\zeta) \coloneqq \int_X  \proj{\bounds}{f - \zeta} \dx \in \R
		.
	\end{equation*}
	First we prove that there exists $\zeta \in \R$ such that the equation $a(\zeta) = C$ is solvable.
	Indeed, it is easy to see that $a(\zeta)$ is Lipschitz with constant $\abs{X}^{1/2}$ and monotone decreasing.
	Moreover, $\lim_{\zeta \to \infty} a(\zeta) = 0$ and $\lim_{\zeta \to -\infty} a(\zeta) = \abs{X}$ hold, and thus $a(\zeta) = C$ is solvable in view of assumption \eqref{eq:assumption_C}.

	Define $\bar v \coloneqq \proj{\bounds}{f - \zeta}$.
	In order to verify that $\bar v = v$ holds, it is necessary and sufficient to prove that \eqref{eq:projection_1} is verified with $v$ replaced by $\bar v$.
	To this end, $w \in \Delta_C \cap \bounds$ be arbitrary.
	Then
	\begin{equation*}
		\begin{aligned}
			\MoveEqLeft
			\inner[auto]{\bar v - f}{w - \bar v}_{L^2(X)}
			\\
			& 
			= 
			\inner[auto]{\proj{\bounds}{f - \zeta} - f}{w - \bar v}_{L^2(X)}
			\quad
			\text{by definition of $\bar v$}
			\\
			& 
			= 
			\inner[auto]{\proj{\bounds}{f - \zeta} - (f - \zeta)}{w - \bar v}_{L^2(X)}
			\quad
			\text{since $\int_X \bar v \dx = \int_X w \dx = C$}
			\\
			& 
			= 
			\int_{\mrep{\setDef{x \in X}{f - \zeta \le 0}}{\hspace{15mm}}} - (f - \zeta) \, (w-0) \dx 
			+ \int_{\mrep{\setDef{x \in X}{f - \zeta \ge 1}}{\hspace{15mm}}} (1- (f - \zeta)) \, (w-1) \dx 
			.
		\end{aligned}
	\end{equation*}
	Since $0 \le w \le 1$ holds, both integrands are non-negative, which proves the claim.

	\Cref{item:projection_2}:
	Suppose that $(v,\zeta) \in L^2(X) \times \R$ satisfies \eqref{eq:proj_Delta_bounds}.
	By the same argument as above (with $v$ in place of $\bar v$) we obtain that \eqref{eq:projection_1} holds.
\end{proof}

We are now in the position to prove necessary and sufficient optimality conditions for the problem \eqref{eq:OED_regularized}--\eqref{eq:bounded_in_L2}, or equivalently, \eqref{eq:decomposition_into_F_G_H_primal}--\eqref{eq:definitions_of_G_H_primal}.
Notice that the symbol $\nabla$ denotes the gradient w.r.t.\ the inner products in $L^2(X)$ and $\S^n$, respectively.

\begin{theorem}[Necessary and sufficient optimality conditions] \hfill
	\label{theorem:optimality_conditions}
	\begin{enumerate}[label=$(\alph*)$]
		\item 
			\label{item:optimality_conditions_1}
			$w \in L^2(X)$ solves \eqref{eq:decomposition_into_F_G_H_primal}--\eqref{eq:definitions_of_G_H_primal} if and only if 
			\begin{equation}
				0 \in \Lambda^* \nabla F_q(\Lambda w) + \alpha \, w  + \partial G(w)
				\label{eq:convex_optimality_conditions}
			\end{equation}
			holds.

		\item 
			\label{item:optimality_conditions_2}
			Suppose $w \in L^2(X)$ fulfills \eqref{eq:convex_optimality_conditions}.
			Then there exists $\zeta \in \R$ such that $(w,\zeta)$ fulfills
			\begin{subequations}
				\begin{align}
					&
					w = \proj[big]{\bounds}{(1-\alpha) w - \Lambda^* \nabla F_q(\Lambda w) - \zeta} \quad \text{a.e.\ in } X
					,
					\label{eq:alternative_optimality_conditions_w}
					\\
					&
					\int_X w \d x = C 
					.
					\label{eq:alternative_optimality_conditions_weight}
				\end{align}
				\label{eq:alternative_optimality_conditions}
			\end{subequations}
			Conversely if $(w,\zeta) \in L^2(X) \times \R$ fulfills \eqref{eq:alternative_optimality_conditions}, then $w$ solves \eqref{eq:decomposition_into_F_G_H_primal}--\eqref{eq:definitions_of_G_H_primal}. 

		\item 
			\label{item:optimality_conditions_4}
			Suppose that $\alpha = 0$ holds and that $w \in L^2(X)$ solves \eqref{eq:decomposition_into_F_G_H_primal}--\eqref{eq:definitions_of_G_H_primal}, i.e., $w$ fulfills \eqref{eq:convex_optimality_conditions}.
			Then there exists another optimal $\bar w \in L^2(X)$ and a measurable set $M \subset X$ such that 
			\begin{equation}
				\label{eq:bang_bang_unregularized}
				\bar w = 1 \quad \text{a.e.\ in $M$}, 
				\quad
				\bar w = 0 \quad \text{a.e.\ in $X \setminus M$}
				.
			\end{equation}
	\end{enumerate}
\end{theorem}
We refer to $\bar w$ as in \eqref{eq:bang_bang_unregularized} as a solution of bang-bang type.
\begin{proof}
	\Cref{item:optimality_conditions_1}:
	It is clear that $w \in L^2(X)$ solves \eqref{eq:decomposition_into_F_G_H_primal} if and only if $0 \in \partial(F(\Lambda w) + (G+H_\alpha)(w))$ holds; see for instance \cite[Thm.~2.5.7]{Zalinescu2002}.
	Since we assumed that there exists $w \in \Delta_C \cap \bounds$ with $F_q(\Lambda w) < \infty$, and since $(F_q \circ \Lambda)(w)$ is continuous at every $w \in L^2(X)$ with $\Lambda w \succ 0$, we can apply \cite[Thm.~2.8.7~(iii)]{Zalinescu2002} and \cite[Ch.~I, Prop.~5.6]{EkelandTemam1999} to obtain that the necessary and sufficient optimality conditions of \eqref{eq:decomposition_into_F_G_H_primal} 
	They can be written as
	\begin{equation*}
		0 \in \Lambda^* \nabla F_q(\Lambda w) + \nabla H_\alpha (w) + \partial G(w) = \Lambda^* \nabla F_q(\Lambda w) + \alpha\,w + \partial G(w)
		,
	\end{equation*}
which is \eqref{eq:convex_optimality_conditions}.

	\Cref{item:optimality_conditions_2}:
	We prove this statement by virtue of the following:
	\begin{align}
		&
		0 \in \Lambda^* \nabla F_q(\Lambda w) + \alpha \, w  + \partial G(w)
		\notag
		\tag{\ref{eq:convex_optimality_conditions}}
		\\
		\Leftrightarrow
		\quad
		&
		(1-\alpha) w - \Lambda^* \nabla F_q(\Lambda w) \in \paren[big](){\id + \partial G}(w)
		\tag{$*$}
		\label{eq:optimality_conditions_2_proof1}
		\\
		\Leftrightarrow
		\quad
		&
		w = \paren[big](){\id +  \partial G}^{-1} \paren[big](){(1-\alpha) w - \Lambda^* \nabla F_q(\Lambda w)}
		\tag{$**$}
		\label{eq:optimality_conditions_2_proof2}
		\\
		\Leftrightarrow
		\quad
		&
		w = \proj[big]{\Delta_C \cap \bounds}{(1-\alpha) w - \Lambda^* \nabla F_q(\Lambda w)}
		.
		\tag{$**$$*$}
		\label{eq:optimality_conditions_2_proof3}
	\end{align}
	The equivalence of \eqref{eq:convex_optimality_conditions} and \eqref{eq:optimality_conditions_2_proof1} is clear.
	The equivalence of \eqref{eq:optimality_conditions_2_proof1} and \eqref{eq:optimality_conditions_2_proof2} follows since $\partial G$ is maximally monotone; see for instance \cite{Minty1964,Rockafellar1976}.
	The equivalence of \eqref{eq:optimality_conditions_2_proof2} and \eqref{eq:optimality_conditions_2_proof3} follows since $G$ is the indicator function of $\Delta_C \cap \bounds$.

	Now if $w \in L^2(X)$ fulfills \eqref{eq:convex_optimality_conditions} and thus \eqref{eq:optimality_conditions_2_proof3}, then \cref{lemma:projection}~\cref{item:optimality_conditions_1} implies the existence of $\zeta$ such that \eqref{eq:alternative_optimality_conditions} holds.
	Conversely, if $(w,\zeta)\in L^2(X) \times \R$ fulfills \eqref{eq:alternative_optimality_conditions}, then \cref{lemma:projection}~\cref{item:optimality_conditions_2} shows that $w$ also solves \eqref{eq:optimality_conditions_2_proof3} and thus also \eqref{eq:convex_optimality_conditions}.

	\Cref{item:optimality_conditions_4}:
	The statement follows directly from \cite[Thm.~1]{Wynn1982}.
\end{proof}

We continue to derive further properties following from the optimality conditions in \cref{theorem:optimality_conditions}.
The following corollary provides a pointwise relation between an optimal $w \in L^2(X)$ and $\Lambda^* \nabla F_q(\Lambda w) \in L^2(X)$.
This allows us to define, on the one hand, a convenient stopping criterion.
On the other hand, it can be used to deduce a certain spatial regularity of an optimal weight~$w$, provided that $\Upsilon$ is more regular than $L^2(X;\S^n)$ and belongs, e.g., to the Hölder class $C^{0,\lambda}$ or Sobolev class $W^{1,p}$.

\begin{corollary}[Further properties of optimal solutions] \hfill
	\label{corollary:regularity_of_w}
	\begin{enumerate}[label=$(\alph*)$]
		\item 
			\label{item:regularity_of_w_1}
			Condition \eqref{eq:alternative_optimality_conditions_w} for  $(w,\zeta)$ is equivalent to the existence of a partion $X_0 \cupdot X_{01} \cupdot X_1 = X$ such that
			\begin{equation}
				\label{eq:w_pointwise}
				\left\{
					\begin{alignedat}{3}
						w(x) = 0 
						\text{ and }
						&
						-\paren[auto](){\Lambda^* \nabla F_q(\Lambda w)}(x) 
						&&
						\le   \zeta 
						&&
						\quad
						\text{a.e.\ in } X_0
						\\
						0 < {}
						w(x) < 1
						\text{ and }
						&
						-\paren[auto](){\Lambda^* \nabla F_q(\Lambda w)}(x) - \alpha\, w(x) 
						&&
						= \zeta
						&&
						\quad
						\text{a.e.\ in } X_{01}
						\\
						w(x) =  1
						\text{ and }
						&
						-\paren[auto](){\Lambda^* \nabla F_q(\Lambda w)}(x) - \alpha 
						&&
						\ge \zeta
						&&
						\quad
						\text{a.e.\ in } X_1
					\end{alignedat}
				\right\}
			\end{equation}

		\item 
			\label{item:regularity_of_w_2}
			Suppose $\alpha > 0$ and $w \in L^2(X)$ fulfills \eqref{eq:convex_optimality_conditions}.
			Then there exists $\zeta \in \R$ such that $(w,\zeta)$ fulfills
				\begin{subequations}
					\begin{align}
						&
						w = \proj[auto]{\bounds}{-\alpha^{-1} \Lambda^* \nabla F_q(\Lambda w) - \zeta} \quad \text{a.e.\ in } X
						,
						\label{eq:alternative_optimality_conditions_regularized_w}
						\\
						&
						\int_X w \d x = C 
						.
						\label{eq:alternative_optimality_conditions_regularized_weight}
					\end{align}
					\label{eq:alternative_optimality_conditions_regularized}
				\end{subequations}
				Conversely if $(w,\zeta) \in L^2(X) \times \R$ fulfills \eqref{eq:alternative_optimality_conditions}, then $w$ solves \eqref{eq:decomposition_into_F_G_H_primal}--\eqref{eq:definitions_of_G_H_primal}. 

		\item 
			\label{item:regularity_of_w_3}
			Suppose $\alpha > 0$ and $\Upsilon \in C^{0,\lambda}(X;\S^n)$ for some $\lambda \in [0,1]$.
			Then the optimal solution \eqref{eq:decomposition_into_F_G_H_primal}--\eqref{eq:definitions_of_G_H_primal} satisfies $w \in C^{0,\lambda}(X)$. 

		\item 
			\label{item:regularity_of_w_4}
			Suppose $\alpha > 0$ and $\Upsilon \in W^{1,p}(X;\S^n)$ for some $p \in [1,\infty]$.
			Then the optimal solution \eqref{eq:decomposition_into_F_G_H_primal}--\eqref{eq:definitions_of_G_H_primal} satisfies $w \in W^{1,p}(X)$.
	\end{enumerate}
\end{corollary}
\begin{proof}
	\Cref{item:regularity_of_w_1}:
	First we consider that $w \in L^2(X)$ fulfills \eqref{eq:alternative_optimality_conditions_w}.
	Clearly $w \in \bounds$ implies the existence of a partition of~$X$ such that
	\begin{equation*}
		\begin{aligned}
			X_0 
			&
			=
			\setDef{x \in X}{w(x) = 0 \text{ a.e.}},
			\\
			X_{01} 
			&
			=
			\setDef{x \in X}{0 < w(x) < 1 \text{ a.e.}},
			\\
			X_1 
			&
			=
			\setDef{x \in X}{w(x) = 1 \text{ a.e.}}.
		\end{aligned}
	\end{equation*}
	It remains to prove the claims for $\paren[auto](){\Lambda^* \nabla F_q (\Lambda w)}(x)$ in \eqref{eq:w_pointwise} on each of these subsets of $X$.

	For $x \in X_0$, \eqref{eq:alternative_optimality_conditions_w} implies
	\begin{alignat*}{2}
		&
		&
		0 
		& 
		\ge
		(1-\alpha)\,w(x) - \paren[auto](){\Lambda^* \nabla F_q (\Lambda w)}(x) - \zeta
		\\
		\Leftrightarrow
		&
		\quad
		&
		0 
		& 
		\ge
		- \paren[auto](){\Lambda^* \nabla F_q (\Lambda w)}(x) - \zeta
		\\
		\Leftrightarrow
		&
		\quad
		&
		\zeta 
		& 
		\ge
		- \paren[auto](){\Lambda^* \nabla F_q (\Lambda w)}(x)
	\end{alignat*}
	For $x \in X_{01}$, we have $0 < w(x) < 1$ and \eqref{eq:alternative_optimality_conditions_w} implies
	\begin{alignat*}{2}
		&
		&
		w(x) 
		& 
		=
		(1-\alpha)\,w(x) - \paren[auto](){\Lambda^* \nabla F_q (\Lambda w)}(x) - \zeta
		\\
		\Leftrightarrow
		&
		\quad
		&
		\zeta
		& 
		=
		- \paren[auto](){\Lambda^* \nabla F_q (\Lambda w)}(x) -  \alpha\,w(x)
		.
	\end{alignat*}
	For $x \in X_1$, the proof is analogous to the case $x \in X_0$.

	Now conversly we assume that a partition of~$X$ exists which fulfills \eqref{eq:w_pointwise}.
	By plugging in the corresponding relations for each subset of this partition, \eqref{eq:alternative_optimality_conditions_w} can be checked easily.

	\Cref{item:regularity_of_w_2}:
	Similarily to the proof of \cref{theorem:optimality_conditions} \cref{item:optimality_conditions_2} we prove this statement by virtue of the following:
	\begin{align*}
		&
		0 \in \Lambda^* \nabla F_q(\Lambda w) + \alpha \, w  + \partial G(w)
		\tag{\ref{eq:convex_optimality_conditions}}
		\\
		\Leftrightarrow
		\quad
		&
		\alpha^{-1} \Lambda^* \nabla F_q(\Lambda w) \in \paren[big](){\id + \alpha^{-1}\partial G}(w)
		\\
		\Leftrightarrow
		\quad
		&
		w = \paren[big](){\id +  \alpha^{-1} \partial G}^{-1} \paren[big](){\alpha^{-1} \Lambda^* \nabla F_q(\Lambda w)}
		\\
		\Leftrightarrow
		\quad
		&
		w = \proj[auto]{\Delta_C \cap \bounds}{\alpha^{-1} \Lambda^* \nabla F_q(\Lambda w)}
		.
	\end{align*}
	The proof is concluded by the same arguments as in the proof of \cref{theorem:optimality_conditions} \cref{item:optimality_conditions_2} and the observation that the indicator function $G$ is invariant under scaling.

	\Cref{item:regularity_of_w_3}:
	Recalling \eqref{eq:alternative_optimality_conditions_regularized_w} and the representation \eqref{eq:synthesis_operator_with_density_Hilbert_adjoint} of $\Lambda^*$, we have 
	\begin{equation*}
		w 
		=
		\proj[auto]{\bounds}{-\alpha^{-1} \Lambda^* \nabla F_q(\Lambda w) - \zeta} \quad \text{a.e.\ in } X
		.
	\end{equation*}
	Suppose now that $\Upsilon \in C^{0,\lambda}(X;\S^n)$ holds, then $- \nabla F_q(\Lambda w) \colon \Upsilon(x)$ belongs to $C^{0,\lambda}(X;\R)$.
	Since $\proj{\bounds}{\cdot}$ preserves the class of Hölder continuous functions, the claim follows.

	\Cref{item:regularity_of_w_4}:
	Similarly as in \cref{item:regularity_of_w_3}, this result follows from the Stampacchia lemma; see for instance \cite[Theorem~A.1]{KinderlehrerStampacchia1980}.
\end{proof}

Notice that \eqref{eq:w_pointwise} implies the \emph{sparsity} of the optimal weight $w$.
This property is reminiscent of optimal control problems with sparsity promoting $L^1$-norm objectives; see for instance \cite[Section~2]{Stadler2007:1} and \cite[Remark~3.3]{CasasHerzogWachsmuth2010:1}. 

For practical purposes, we also discuss a relaxed version of the optimality conditions \eqref{eq:convex_optimality_conditions}.
\begin{lemma}[Relaxed optimality conditions]
	\label{lemma:relaxed_optimality_conditions}
	Let $\varepsilon > 0$.
	The following statements are equivalent:
	\begin{enumerate}[label=$(\alph*)$]
		\item 
			\label{item:relaxed_optimality_conditions_1}
			$w \in L^2(X)$ solves the relaxed optimality conditions 
			\begin{equation}
				y \in \Lambda^* \nabla F_q (\Lambda w ) + \alpha\,w + \partial G(w) 
				\label{eq:convex_optimality_conditions_relaxed}
			\end{equation}
			with some $y \in L^{\infty}(X)$ satisfying $\norm{y}_{L^\infty(X)} \le \varepsilon$.
		\item 
			\label{item:relaxed_optimality_conditions_2}
			There exists $w \in \Delta_C$ and $\zeta \in \R$ as well as a partition $X_0 \cupdot X_{01} \cupdot X_1 = X$  such that
			% \begin{equation}
			%   \label{eq:w_pointwise_relaxed_estimated}
			%   \left\{
			%     \begin{alignedat}{4}
			%       w(x) = 0 
			%       &
			%       \text{ and }
			%       &
			%       &
			%       -\paren[auto](){\Lambda^* \nabla F_q(\Lambda w)}(x) 
			%       &&
			%       \le \zeta +\varepsilon
			%       &&
			%       \quad
			%       \text{a.e.\ in } X_0
			%       \\
			%       0 < {}
			%       w(x) 
			%       < 1
			%       &
			%       \text{ and }
			%       &
			%       \zeta  - \varepsilon 
			%       \le
			%       &
			%       -\paren[auto](){\Lambda^* \nabla F_q(\Lambda w)}(x) - \alpha\, w(x)
			%       &&
			%       \le \zeta+ \varepsilon
			%       &&
			%       \quad
			%       \text{a.e.\ in } X_{01}
			%       \\
			%       w(x) =  1
			%       &
			%       \text{ and}
			%       &
			%       \zeta  -\varepsilon 
			%       \le
			%       &
			%       -\paren[auto](){\Lambda^* \nabla F_q(\Lambda w)}(x) - \alpha
			%       &&
			%       &&
			%       \quad
			%       \text{a.e.\ in } X_1
			%     \end{alignedat}
			%   \right\}
			% \end{equation}
			\begin{equation}
				\label{eq:w_pointwise_relaxed_estimated}
				\left\{
					\begin{alignedat}{4}
						w(x) =  {}
						& 0 \text{ and }
						&&
						&&
						&&
						\multirow{2}{*}{$\;\;\text{a.e.\ in } X_0$}
						\\
						&
						&&
						-\paren[auto](){\Lambda^* \nabla F_q(\Lambda w)}(x) 
						&&
						\le \zeta +\varepsilon
						\\
						0 < w(x) < {}
						& 
						1 \text{ and }
						&&
						&&
						&&
						\multirow{2}{*}{$\;\;\text{a.e.\ in } X_{01}$}
						\\
						&
						\zeta  - \varepsilon 
						\le
						&&
						-\paren[auto](){\Lambda^* \nabla F_q(\Lambda w)}(x) - \alpha\, w(x)
						&&
						\le \zeta+ \varepsilon
						\\
						w(x) = {} 
						& 
						1	\text{ and}
						&&
						&&
						&&
						\multirow{2}{*}{$\;\;\text{a.e.\ in } X_{1}$}
						\\
						&
						\zeta  -\varepsilon 
						\le
						&&
						-\paren[auto](){\Lambda^* \nabla F_q(\Lambda w)}(x) - \alpha
					\end{alignedat}
				\right\}
			\end{equation}

	\end{enumerate}
\end{lemma}
\begin{proof}
	Similarly as in the proof of \cref{theorem:optimality_conditions} \cref{item:optimality_conditions_2} one can show that if $w \in L^2(X)$ fulfills \eqref{eq:convex_optimality_conditions_relaxed}, then there exists $\zeta \in \R$ such that $(w,\zeta)$ verifies
	\begin{subequations}
		\begin{align}
			&
			w = \proj[auto]{\bounds}{(1-\alpha) w  + y - \Lambda^* \nabla F_q(\Lambda w) - \zeta} \quad \text{a.e.\ in } X
			,
			\label{eq:alternative_optimality_conditions_relaxed_w}
			\\
			&
			\int_X w \d x = C 
			.
			\label{eq:alternative_optimality_conditions_relaxed_weight}
		\end{align}
		\label{eq:alternative_optimality_conditions_relaxed}
	\end{subequations}
	Conversely, if $(w,\zeta)$ fulfills \eqref{eq:alternative_optimality_conditions_relaxed}, then $w$ also fulfills \eqref{eq:convex_optimality_conditions_relaxed}.

	Analogously to \cref{corollary:regularity_of_w} \cref{item:regularity_of_w_1}, we find that \eqref{eq:alternative_optimality_conditions_relaxed_w} for $(w,\zeta)$ is equivalent to the existence of a partition $X_0 \cupdot X_{01} \cupdot X_1 = X$ such that 
	\begin{equation}
		\label{eq:w_pointwise_relaxed}
		\left\{
			\begin{alignedat}{3}
				w(x) = 0 
				\text{ and }
				&
				-\paren[auto](){\Lambda^* \nabla F_q(\Lambda w)}(x) 
				&&
				\le   \zeta - y
				&&
				\quad
				\text{a.e.\ in } X_0
				\\
				0 < {}
				w(x) < 1
				\text{ and }
				&
				-\paren[auto](){\Lambda^* \nabla F_q(\Lambda w)}(x) - \alpha\, w(x) 
				&&
				= \zeta -y 
				&&
				\quad
				\text{a.e.\ in } X_{01}
				\\
				w(x) =  1
				\text{ and }
				&
				-\paren[auto](){\Lambda^* \nabla F_q(\Lambda w)}(x) - \alpha 
				&&
				\ge \zeta- y
				&&
				\quad
				\text{a.e.\ in } X_1
			\end{alignedat}
		\right\}
	\end{equation}

	It remains to prove that $(w,\zeta)$ fulfills \eqref{eq:w_pointwise_relaxed} for $\norm{y}_{L^\infty(X)} \le \varepsilon$ if and only if $(w,\zeta)$ fulfills \eqref{eq:w_pointwise_relaxed_estimated}.
	Suppose first that \eqref{eq:w_pointwise_relaxed} holds for some $y$ with $\norm{y}_{L^\infty(X)} \le \varepsilon$.
	Then clearly \eqref{eq:w_pointwise_relaxed_estimated} holds.
	Conversely, suppose that \eqref{eq:w_pointwise_relaxed_estimated} holds and choose
	\begin{equation*}
		y
		\coloneqq
		\begin{cases}
			-\varepsilon
			&
			\text{ on } X_0,
			\\
			-\paren[auto](){\Lambda^* \nabla F_q(\Lambda w)} -  \zeta - \alpha\, w
			&
			\text{ on } X_{01},
			\\
			\varepsilon
			&
			\text{ on } X_1
			.
		\end{cases}
	\end{equation*}
	By \eqref{eq:w_pointwise_relaxed_estimated}, we have $\norm{y}_{L^\infty(X)} \le \varepsilon$.
	This shows \eqref{eq:w_pointwise_relaxed}.
\end{proof}

We conclude this section by providing the formula for the gradients of the design criteria~$F_q$ from \eqref{eq:criteria}.
Its proof can be found, e.g., in \cite[eqs.~(51) and  (110)]{PetersenPedersen2008}.
\begin{lemma}
	\label{lemma:gradients}
	The gradient of $F_q$ at $\FIM \in \S^n$ such that $\FIM \succ 0$ is given by
	\begin{equation}
		\label{eq:gradient}
		\nabla F_q(\FIM) =
		\begin{cases}
			- (\frac{1}{n})^{\frac{1}{q}} (\trace \FIM^{-q})^{\frac{1}{q}-1} \, I^{-q-1}
			&
			\text{for } q \in (0, \infty),
			\\
			- \frac{1}{n} \FIM^{-1}
			&
			\text{for } q = 0
			.
		\end{cases}
	\end{equation}
\end{lemma}

%------------------------------------------------------------------
\section{Discretization by Piecewise Constants}
\label{sec:Discretization_DG0}
%------------------------------------------------------------------

In this section we discuss the discretization of the design space $X$ and the design density $w \in L^2(X)$.
Suppose that $X$ can be discretized into a finite number~$m$ of compact polyhedra (cells) $(E_i)_{i = 1,\dots, m}$ of dimension~$d$, whose interiors do not intersect.
Based on this discretization, $w$ is assumed to be $DG_0$, i.e., constant on each $E_i$.
We write 
\begin{equation*}
	w(x) = \sum_{i = 1}^m w_i \, \chi_{E_i}(x)
\end{equation*}
where $\chi_{E_i}$ denotes the indicator function of cell~$E_i$. 
We endow the space $\R^m$ of coefficients of $w$ with the appropriate inner product represented by the matrix
\begin{equation*}
	M 
	\coloneqq 
	\diag(\abs{E_1}, \ldots, \abs{E_m})
	,
\end{equation*}
where $\abs{E_i}$ denotes the volume of the $i$-th cell.
We associate with each cell the elementary experiment $\Upsilon_i$ whose FIM is given by some approximation of the average FIM $\int_{E_i} \Upsilon(\cdot) \dx / \abs{E_i}$.
When $\Upsilon \in C^0(X;\S^n)$, we can simply take $\Upsilon_i = \Upsilon(x_i)$, where $x_i$ is the midpoint of $E_i$.
This is our method of choice for the numerical experiments in \cref{sec:numerical_results,sec:simplicial_decomposition} .

Consequently, we represent the synthesis operator $\Lambda: \R^m \to \S^n$ as
\begin{equation}
	\label{eq:synthesis_operator_discrete}
	\Lambda w 
	= 
	\sum_{i = 1}^m \Upsilon_i \, \abs{E_i} \, w_i \in \S^n 
	.
\end{equation}
Its Hilbert space adjoint $\Lambda^*: \S^n \to \R^m$ (w.r.t.\ the $M$-inner product in $\R^m$) is then given by
\begin{equation}
	[\Lambda^* P]_i = P \colon \Upsilon_i
	,
	\quad
	i = 1, \ldots, m
	.
\end{equation}
The bound constraints are now written as 
\begin{equation}
	\label{eq:bounded_in_L2_discrete}
	\bounds 
	\coloneqq 
	[0,1]^m 
	=
	\setDef{w \in \R^m}{0 \le w_i \le 1 \text{ for all } i = 1, \ldots, m}
\end{equation}
and the total mass constraint becomes
\begin{equation}
	\label{eq:simplex_in_L2_discrete}
	\Delta_C 
	\coloneqq 
	\setDef[auto]{w \in \R^m}{w \ge 0 \text{ and } \sum_{i=1}^m \abs{E_i} \, w_i = C} 
	.
\end{equation}
Hence the discrete version of problem \eqref{eq:OED_regularized}--\eqref{eq:bounded_in_L2} can be cast as
\begin{equation}
	\label{eq:discretized_problem}
	\begin{aligned}
		\text{Minimize} \quad & F_q(\Lambda w) + \frac{\alpha}{2} \norm{w}_M^2, \quad w \in \R^m \\
		\text{s.t.} \quad & w \in \Delta_C \cap \bounds
		.
	\end{aligned}
\end{equation}

%------------------------------------------------------------------
\subsection*{Evaluation of $\proj{\Delta_C \cap \bounds}{\cdot}$}
%------------------------------------------------------------------

The orthogonal projection w.r.t.\ the $M$-inner product onto the constraint set $\Delta_C \cap \bounds$ is an essential ingredient in all forthcoming algorithms.
We begin by stating a discrete version of \cref{lemma:projection} without proof.
\begin{lemma}
	\label{lemma:projection_discrete}
	Let $f \in \R^m$.
	\begin{enumerate}[label=$(\alph*)$]
		\item 
			Suppose that $v = \proj{\Delta_C \cap \bounds}{f}$ holds.
			Then there exists $\zeta \in \R$ such that 
			\begin{equation}
				\label{eq:proj_Delta_bounds_discrete}
				v = \proj{\bounds}{f - \zeta} 
				\quad \text{and} \quad
				\sum_{i=1}^m \abs{E_i} \, v_i = C
			\end{equation}
			holds.

		\item 
			Conversely, suppose that $(v,\zeta) \in \R^m \times \R$ satisfies \eqref{eq:proj_Delta_bounds_discrete}, then we have $v = \proj{\Delta_C \cap \bounds}{f}$.
	\end{enumerate}
\end{lemma}
Here $f - \zeta$ denotes the vector with components $f_i - \zeta$.
We denote a solution of \eqref{eq:proj_Delta_bounds_discrete} by $(v^*,\zeta^*)$. 
Notice that $v^*$ is unique but in general this does not hold for $\zeta^*$.

In the remainder of this section we describe an efficient algorithm to evaluate $\proj{\Delta_C \cap \bounds}{f}$.
Without loss of generality, we assume that the entries of $f$ are sorted in descending order.
In order to facilitate the notation, we introduce a number of abbreviations, where $v \in \R^m$, $\zeta \in \R$ and $k, \ell \in \{1, \ldots, m\}$ are arbitrary:
\begin{align}
	\nnz(v) 
	&
	\coloneqq \abs[big]{\setDef{ v_i}{i \in \{ 1, \dots, m \} \text{ and } v_i > 0}}
	&
	&
	\text{\#non-zeros in $v$}
	,
	\nonumber
	\\
	\neo(v) 
	&
	\coloneqq \abs[big]{\setDef{ v_i}{i \in \{ 1, \dots, m \} \text{ and } v_i = 1}}
	&
	&
	\text{\#ones in $v$}
	,
	\nonumber
	\\
	k^* 
	&
	\coloneqq \nnz(v^*)
	\quad
	\text{and}
	\quad
	\ell^* 
	\coloneqq \neo(v^*)
	&
	&
	\text{\#non-zeros, ones in $v^*$}
	,
	\nonumber
	\\
	\R \ni \zeta \mapsto P(\zeta) 
	& 
	\coloneqq \proj{\bounds}{f-\zeta}
	\in \R^m
	,
	\nonumber
	\\
	\R^m \ni v \mapsto W(v) 
	&
	\coloneqq
	\sum_{i=1}^m \abs{E_i} v_i
	\in \R
	.
	\label{eq:nnz_neo_and_others}
\end{align}
Notice that once the optimal number of ones $\ell^*$ and non-zeros $k^*$ are known, the evaluation of the shift $\zeta^*$ and thus $v^*$ becomes trivial.
Indeed, the total weight constraint reads
\begin{equation}
	\label{eq:condition_for_zeta}
	\sum_{i=1}^{\ell^*} \abs{E_i}
	+
	\sum_{i=\ell^* + 1}^{k^*} (f_i - \zeta^*) \, \abs{E_i}
	=
	C
	.
\end{equation}
The solution of this constraint w.r.t.\ $\zeta$ leads to the definition of the following function,
\begin{equation}
	\label{eq:definition_of_zeta_function}
	\zeta(k,\ell) 
	\coloneqq 
	\begin{cases}
		\displaystyle\paren[auto](){\sum_{i = \ell+1}^k \abs{E_i}}^{-1} \paren[auto](){\sum_{i = 1}^{\ell} \abs{E_i} + \sum_{i=\ell+1}^k \abs{E_i} \, f_i - C}
		&
		\text{if } k > \ell,
		\\
		f_k-1
		&
		\text{if } k = \ell,
	\end{cases}
\end{equation}
for $1 \le \ell \le k \le m$.
The second case selects the largest possible value of $\zeta$ compatible with the assumption $k = \ell$, i.e., that all positive entries of $\proj{\bounds}{f - \zeta(k,\ell)}$ are indeed equal to one.

Before we present the idea to evaluate $v^* = \proj{\Delta_C \cap \bounds}{f}$ efficiently, we characterize the (non\nobreakdash-)uniqueness of $\zeta^*$.
\begin{lemma}
	\label{lemma:uniqueness_zetastar}
	\begin{enumerate}[label=$(\alph*)$]
		\item 
			\label{item:uniqueness_zetastar_lstar_equal_kstar}
			Suppose that $\ell^* = k^*$ holds.
			Then $f_{k^*+1} \le f_{k^*} - 1$.
			Moreover, \eqref{eq:proj_Delta_bounds_discrete} is valid if and only if $\zeta^* \in [f_{k^*+1}, f_{k^*} -1]$.
		\item 
			\label{item:uniqueness_zetastar_lstar_smaller_kstar}
			Suppose that $\ell^* < k^*$.
			Then \eqref{eq:proj_Delta_bounds_discrete} is valid if and only if $\zeta^* = \zeta(k^*,\ell^*)$.
	\end{enumerate}
\end{lemma}
\begin{proof}
	\Cref{item:uniqueness_zetastar_lstar_equal_kstar}:
	Since $\ell^* = k^*$ holds, we obviously have $f_{k^*} -\zeta^* \ge 1$ and $f_{k^* +1 } - \zeta^* \le 0$ for any $(v^*,\zeta^*)$ satisfying \eqref{eq:proj_Delta_bounds_discrete}.
	Therefore $\zeta^*$ is bounded by 
	\begin{equation*}
		f_{k^*} -1 \ge \zeta^* \ge f_{k^* + 1}.
	\end{equation*}
	This proves the first claim and the 'only if' part.

	Now we choose $\zeta \in [f_{k^* +1}, f_{k^*} -1]$ arbitrarily.
	For all $i \le k^*$ we get
	\begin{equation*}
		f_i - \zeta \ge f_{k^*} - \zeta \ge f_{k*} - f_{k^*} +1  = 1 
	\end{equation*}
	and analogously for all $i \ge k^*+1$ we have
	\begin{equation*}
		f_i - \zeta \le f_{k^* +1} - \zeta \le f_{k^* +1} - f_{k^* +1} = 0
		.
	\end{equation*}
	So, since $k^* = \ell^*$ holds, we have $P(\zeta) = v^*$.

	\Cref{item:uniqueness_zetastar_lstar_smaller_kstar}:
	When $\ell^* < k^*$, then necessarily $0 < v^*_{\ell^*+1} = f_{\ell^*+1} - \zeta^* < 1$.
	Since $v^*$ is unique, $\zeta^*$ is unique as well.
	In view of \eqref{eq:condition_for_zeta}, $\zeta^* = \zeta(k^*,\ell^*)$ is necesary and sufficient for \eqref{eq:proj_Delta_bounds_discrete}.
\end{proof}

Based on these results, we propose to compute $v^* = \proj{\Delta_C \cap \bounds}{f}$ by the following algorithm.
\begin{algorithm}[ht]
	\caption{Computation of $\proj{\Delta_C \cap \bounds}{f}$}
	\label{alg:projection_feasible_set}
	\begin{algorithmic}[1]
		\setcounter{ALC@unique}{0}
		\REQUIRE $f \in \R^m$ (with entries sorted in descending order of magnitude)
		\REQUIRE weights $\abs{E_i}$ (with entries sorted accordingly)
		\ENSURE $v \in \R^m$, where $v = \proj{\Delta_C \cap \bounds}{f}$
		\STATE Find $k  \coloneqq \max \setDef{i \in \{1, \dots, m\}}{W(P(f_i)) < C}$
		\label{alg:projection_feasible_set_definition_k}
		\STATE Compute $\ell_{\min} \coloneqq \abs[big]{\setDef{i \in \{1, \dots, k\}}{f_i \ge f_k + 1}} \ge 0$
		\label{alg:projection_feasible_set_definition_lmin}
		\STATE Compute 
		\label{alg:projection_feasible_set_definition_lmax}
		\begin{equation*}
			\ell_{\max} 
			\coloneqq 
			\begin{cases}
				\abs[big]{\setDef{i \in \{1, \dots, k\}}{f_i \ge f_{k+1} + 1}}
				&
				\text{if $k < m$},
				\\
				m-1
				&
				\text{if $k = m$}.
			\end{cases}
		\end{equation*}
		\STATE Find $\ell \coloneqq \min \setDef{i \in \{\ell_{\min}, \dots, \ell_{\max} \}}{1 > f_{i+1} - \zeta(k,i)}$
		\label{alg:projection_feasible_set_definition_l}
		\STATE Set $v \coloneqq P( \zeta(k,\ell))$ 
		\label{alg:projection_feasible_set_generate_v}
	\end{algorithmic}
\end{algorithm}

Before we can prove the correctness of \cref{alg:projection_feasible_set}, we need the following technical lemma, whose proof is given in \cref{section:proof_lemma:technical_results}.
\begin{lemma} 
	\label{lemma:technical_results}
	The following statements are true.
	\begin{enumerate}[label=$(\alph*)$]
		\item 
			$\zeta \mapsto W(P(\zeta))$ is monotonically decreasing.
			\label{item:monotonicity_W_P}

		\item 
			For $0 \le \ell \le k^*-2$ we have
			\begin{equation}
				\label{eq:monotonicity_zeta}
				\zeta(k^*,\ell+1) 
				\begin{cases}
					< \zeta(k^*,\ell)
					&
					\text{if } 1 < f_{\ell+1} - \zeta(k^*,\ell),
					\\
					= \zeta(k^*,\ell)
					&
					\text{if } 1 = f_{\ell+1} - \zeta(k^*,\ell),
					\\
					> \zeta(k^*,\ell)
					&
					\text{if } 1 > f_{\ell+1} - \zeta(k^*,\ell).
				\end{cases}
			\end{equation}
			If $k^* = \ell^*$ we have $\zeta(k^*,k^*) = \zeta(k^*,k^*-1)$.
			\label{item:monotonicity_zeta}

		\item 
			For $k^* = \ell^*$ we have
			\begin{equation*}
				f_{i+1} - \zeta(k^*,i) 
				\begin{cases}
					\ge 1 
					&
					\text{if } i \in \{ 1, \dots, k^*-1 \},
					\\
					< 1 
					&
					\text{if } i = k^*
					.
				\end{cases}
			\end{equation*}
			\label{item:feasibility_lstar_equal_kstar}

		\item 
			Suppose $k^* > \ell^*$.
			If there exists $r \in \{ \ell_{\min}, \dots, k^*-1 \}$ such that $1 > f_{r+1} - \zeta(k^*,r)$ holds, then there is a unique $\ell \in \{ \ell_{\min}, \dots,k^*-1 \}$ such that 
			\begin{equation*} 
				f_{i +1} - \zeta(k^*,i) 
				\begin{cases}
					\ge 1 
					&
					\text{if } i \in \{ \ell_{\min}, \dots, \ell -1 \},
					\\
					< 1 
					&
					\text{if } i \in \{\ell, \dots, k^*-1 \}.
				\end{cases}
			\end{equation*}
			\label{item:feasibility_lstar_smaller_kstar}
	\end{enumerate}
\end{lemma}

We are now in a position to prove the correctness of \cref{alg:projection_feasible_set}.
\begin{theorem}
	\label{theorem:correctness_of_alg:projection_feasible_set}
	For arbitrary $f \in \R^m$, the pair $(v,\zeta)$ as defined in \cref{alg:projection_feasible_set}, satifies \eqref{eq:proj_Delta_bounds_discrete}, i.e., $v = \proj{\Delta_C \cap \bounds}{f}$ holds.
\end{theorem}
\begin{proof}
	Throughout the proof we assume that $k, \ell, \zeta$ and $v$ are defined as in \cref{alg:projection_feasible_set}.
	As before, we set $v^* = \proj{\Delta_C \cap \bounds}{f}$ and $\zeta^*$ according to \eqref{eq:proj_Delta_bounds_discrete}.
	The proof is organized as follows.
	We first show $k = k^*$ and $\ell = \ell^*$, i.e., \cref{alg:projection_feasible_set} predicts the correct number of zeros and ones in $v^*$.
	This is the main part of the proof.
	\begin{itemize}
		\item 
			In order to show $k = k^*$, we distinguish two cases.
			$(i)$ First we assume $k^* = m$ (the number of weights).
			Obviously we have $f_i > \zeta^*$ for all $i \in \{ 1, \dots, m \}$. So by \cref{lemma:technical_results} \cref{item:monotonicity_W_P} we have $W(P(f_i)) < W(P(\zeta^*)) = W(v^*) = C$.
			So we get $k= m = k^*$, see \cref{alg:projection_feasible_set_definition_k} of \cref{alg:projection_feasible_set}.

			$(ii)$ Now we discuss the case $k^* < m$.
			By definition of $k^*$ we have $f_{k^*} > \zeta^* \ge f_{k^* +1}$.
			By \cref{alg:projection_feasible_set_definition_k} of \cref{alg:projection_feasible_set}, we have 
			\begin{equation*}
				W(P(f_k)) < C \text{ as well as } W(P(f_{k+1})) \ge C.
			\end{equation*}
			Since $W(P(\cdot))$ is monotonically decreasing (cf.\ \cref{lemma:technical_results} \cref{item:monotonicity_W_P}) and $W(P(\zeta^*)) = v^* = C$ holds, we have $f_k > \zeta^* \ge f_{k+1}$.
			Thus $k = k^*$ follows.

		\item 
			$\ell = \ell^*$:
			First we prove $\ell_{\min} \le \ell^* \le \ell_{\max}$ with $\ell_{\min}$ and $\ell_{\max}$ as defined in \cref{alg:projection_feasible_set}.
			To see the first inequality, consider an arbitrary $i \in \{ \ell^* +1, \dots, k^* \}$.
			Then
			\begin{equation*}
				1 > f_i - \zeta^* \quad \Leftrightarrow \quad f_i < 1 + \zeta^* < 1 + f_{k^*}
			\end{equation*}
			holds and thus we have 
			\begin{equation*}
				\ell_{\min} \, \abs[big]{\setDef{i \in \{1, \dots, k\}}{f_i \ge f_k + 1}}
				\le
				k - (k-(\ell^*+1)+1)
				=
				\ell^*.
			\end{equation*}
			To show $\ell^* \le \ell_{\max}$, we distinguish two cases.
			$(i)$ When $k^* = m$, then we have $\ell^* < m$ due to assumption~\eqref{eq:assumption_C}.
			Thus \cref{alg:projection_feasible_set_definition_lmax} of \cref{alg:projection_feasible_set} yields $\ell_{\max} = m-1$.
			$(ii)$ Otherwise we have $k^* < m$.
			By definition of $\ell^*$, we have  
			\begin{equation*}
				1 \le f_i - \zeta^* \quad \Leftrightarrow \quad f_i \ge 1 + \zeta^* \ge 1 + f_{k^*+1}
			\end{equation*}
			for arbitrary $i \in \{ 1, \dots, \ell^* \}$.
			So similarly as in case~$(i)$ we obtain
			\begin{equation*}
				\ell_{\max} 
				= 
				\abs[big]{\setDef{ i \in \{1, \dots, k\}}{f_i \ge f_{k+1} + 1}}
				\ge
				\ell^*.
			\end{equation*}

			Next we show that $1 > f_{\ell^*+1} - \zeta(k^*,\ell^*)$ holds.
			Indeed, we recall $\ell^* < m$ by \eqref{eq:assumption_C}, so that the definition of $\ell^*$ implies $1 > f_{\ell^*+1} - \zeta^*$.
			Again, we need to distinguish two cases.
			$(i)$ If $k^* > \ell^*$, then \cref{lemma:uniqueness_zetastar}~\cref{item:uniqueness_zetastar_lstar_smaller_kstar} shows $\zeta(k^*, \ell^*) = \zeta^*$, which proves the claim.
			$(ii)$ If $k^* = \ell^*$, then $f_{\ell^*+1} - \zeta(k^*,\ell^*) = f_{\ell^*+1} - f_{\ell^*} + 1 < 1$ as desired.

			So far we proved $\ell^* \in L \coloneqq \setDef{ i \in \{ \ell_{\min}, \dots, \ell_{\max} \}}{1  > f_{i+1} - \zeta(k^*,i)}$.
			To conclude the proof of $\min L \eqqcolon \ell = \ell^*$, we begin $(i)$ with the case $k^* > \ell^*$.
			In this case, \cref{lemma:technical_results}~\cref{item:feasibility_lstar_smaller_kstar} shows that there is a unique $\bar \ell \in \{ \ell_{\min}, \dots, \ell_{\max} \}$ such that
			\begin{equation*} 
				f_{i +1} - \zeta(k^*,i) 
				\begin{cases}
					\ge 1 
					&
					\text{if } i \in \{ \ell_{\min}, \dots, \bar \ell -1 \},
					\\
					< 1 
					&
					\text{if } i \in \{\bar \ell, \dots, \ell_{\max} \}
				\end{cases}
			\end{equation*}
			holds.
			Obviously we have $\ell = \min L = \bar \ell$.
			Since we already proved $1 > f_{\ell^*+1} - \zeta(k^*,\ell^*)$, clearly we have $\ell^* \ge \ell$.
			It remains to prove $\ell^* \le \ell$.
			We proceed by contradiction and assume $\ell^* > \ell = \bar \ell$.
			Thus by the property of $\bar \ell$ we have 
			\begin{equation}
				1 > f_{\ell^*} - \zeta(k^*,\ell^*-1)
				.
				\tag{$*$}
				\label{eq:previous_estimate}
			\end{equation}
			By \cref{lemma:technical_results}~\cref{item:monotonicity_zeta} and \eqref{eq:previous_estimate} we have $\zeta(k^*,\ell^*) > \zeta(k^*,\ell^*-1)$.
			Plugging this into \eqref{eq:previous_estimate}, we obtain
			\begin{equation*}
				1 > f_{\ell^*} - \zeta(k^*,\ell^*-1) > f_{\ell^*} - \zeta(k^*,\ell^*) = f_{\ell^*} - \zeta^* = v^*_{\ell^*}
			\end{equation*}
			by \cref{lemma:uniqueness_zetastar}~\cref{item:uniqueness_zetastar_lstar_smaller_kstar}, which contradicts the definition of $\ell^*$.
			Thus we have $\ell = \ell^*$ in case~$(i)$.
			In case~$(ii)$, i.e., $\ell^* = k^*$, we can apply \cref{lemma:technical_results}~\cref{item:feasibility_lstar_equal_kstar}, which shows that $L$ contains only one element.
					Thus $\ell^* = \ell$.
	\end{itemize}
	It remains to conclude that $v$, returned by \cref{alg:projection_feasible_set}, coincides with $v^*$.
	In case $k^* < \ell^*$, we can apply \cref{lemma:uniqueness_zetastar}~\cref{item:uniqueness_zetastar_lstar_smaller_kstar} and we obtain $\zeta^* = \zeta(k^*,\ell^*)$, thus we also get $v \coloneqq P(\zeta(k^*,\ell^*)) = P(\zeta^*) = v^*$.
	In case $k^* = \ell^*$, $\zeta^*$ is not necessarily unique but $\zeta(k^*,\ell^*) = f_{k^*} -1$ is a viable choice by \cref{lemma:uniqueness_zetastar}~\cref{item:uniqueness_zetastar_lstar_equal_kstar}.
\end{proof}

%------------------------------------------------------------------
\section{Numerical Results for Two First-Order Algorithms}
\label{sec:numerical_results}
%------------------------------------------------------------------

In this section we consider two algorithms to find solutions $w \in L^2(X)$ for problem \eqref{eq:decomposition_into_F_G_H_primal}--\eqref{eq:definitions_of_G_H_primal}, along with numerical results.
Specifically, we discuss FISTA (\cite{BeckTeboulle2009}) and the proximal extrapolated gradient method from \cite{Malitsky2017}.
We present them only in the continuous setting since their discrete counterparts are then obvious.
In what follows, we denote the objective by
\begin{equation*}
	a(w) 
	\coloneqq 
	F_q(\Lambda w) + \frac{\alpha}{2}\norm{w}_{L^2(X)}^2
\end{equation*}
and write the optimality system \eqref{eq:convex_optimality_conditions} as 
\begin{equation}
	\label{eq:primal_optimality_conditions_modified}
	0 \in A (w) + \partial G(w)
\end{equation}
where $A(w) \coloneqq \nabla a(w) = \Lambda^* \nabla F_q(\Lambda w) + \alpha \, w$.

An essential operation in both first-order methods under consideration is the evaluation of $\prox{\gamma \partial G}{w}$.
Since in our problem, $G$ is the indicator function of the restricted simplex $\Delta_C \cap \bounds$, we have 
\begin{equation*}
	\prox{\gamma \partial G}{w} = \proj{\Delta_C \cap B}{w},
\end{equation*}
regardless of the value of $\gamma > 0$.
This relation be inserted directly into the algorithms.

%------------------------------------------------------------------
\subsection{FISTA}
\label{subsection:FISTA}
%------------------------------------------------------------------

The well-known fast iterative shrinkage and thresholding algorithms (FISTA) was introduced in \cite{BeckTeboulle2009} and it belongs to the class of extrapolated proximal gradient methods.

\begin{algorithm}[ht]
	\caption{FISTA}
	\label{alg:fista}
	\begin{algorithmic}[1]
		\setcounter{ALC@unique}{0}
		\REQUIRE $w^{(-1)}= w^{(0)} \in L^2(X)$, $\gamma^{(0)} > 0$, $t_1 = 0$ and $\eta \in (0 ,1)$
		\STATE Set $k \coloneqq 0$
		\REPEAT
		\STATE Set $k \coloneqq k + 1$
		\STATE Set $t_{k+1} \coloneqq \frac{1}{2}\paren[auto](){1+ \sqrt{1+4t_k^2}}$
		\STATE Set $\theta^{(k)} \coloneqq \max \paren[auto]\{\}{0, \frac{t_k -1}{t_{k+1}}} $
		\STATE Set $v^{(k)} \coloneqq w^{(k)} + \theta^{(k)} \, (w^{(k)} -w^{(k-1)})$
		\STATE Set $i\coloneqq-1$
		\REPEAT
		\STATE Set $i \coloneqq i +1$
		\STATE Set $\bar \gamma \coloneqq \eta^i \, \gamma^{(k-1)}$
		\STATE Set $y \coloneqq \proj[big]{\Delta_C \cap \bounds}{v^{(k)}- \bar \gamma\, A (v^{(k)})}$
		\UNTIL {$a(y) \le a(v^{(k)}) + \inner{A(v^{(k)})}{y-v^{(k)}}_{L^2(X)} + (2 \bar \gamma)^{-1} \norm{y-v^{(k)}}^2_{L^2(X)}$} 
		\label{alg:fista_gamma}
		\STATE Set $\gamma^{(k)} \coloneqq \bar \gamma$
		\STATE Set $w^{(k+1)} \coloneqq y$
		\label{alg:extrap_prox_gradient_prox}
		\UNTIL{a stopping criterion is fulfilled for $w^{(k+1)}$}
	\end{algorithmic}
\end{algorithm}

A convergence result for a finite dimensional version of \cref{alg:fista} was given in \cite[Theorem~4.4]{BeckTeboulle2009}.
This result assumes a global Lipschitz continuity of $A$, i.e., that there exists $L_A > 0$ such that 
\begin{equation*}
	\norm{A(w_1) - A(w_2)}_{L^2(X)} \le L_A \norm{w_1-w_2}_{L^2(X)}
\end{equation*}
holds for all $w_1,w_2 \in L^2(X)$.
Unfortunately, the gradient of $F_q(\Lambda w)$ exists only on
\begin{equation*}
	\setDef{w \in L^2(X)}{\Lambda w \succ 0}
\end{equation*}
and moreover, it is only locally Lipschitz on this set.
Numerically, however, we did not observe any difficulties associated with this fact since the eigenvalues of the information matrices $\FIM^{(k)} = \Lambda w^{(k)}$ remained bounded away from zero throughout the iterations in our experiments.

%------------------------------------------------------------------
\subsection{Extrapolated Proximal Gradient Method (PGMA)}
\label{subsection:Malitsky}
%------------------------------------------------------------------

In this section we consider an extrapolated proximal gradient method (PGMA) presented by \cite{Malitsky2017}.
Notice that besides the different choice of parameters, this algorithm differs from \cref{alg:fista} mainly by the fact that in \cref{alg:malitsky_prox_gradient_prox} of \cref{alg:malitsky_prox_gradient}, the projection onto $\Delta_C \cap \bounds$ is evaluated in $w^{(k)} - \gamma^{(k)} A (v^{(k)})$ instead of $v^{(k)} - \gamma^{(k)} A (v^{(k)})$.
The convergence for a finite dimensional version of \cref{alg:malitsky_prox_gradient} was shown in \cite[Theorem~3.3]{Malitsky2017}.

\begin{algorithm}[ht]
	\caption{Extrapolated proximal gradient method (PGMA)}
	\label{alg:malitsky_prox_gradient}
	\begin{algorithmic}[1]
		\setcounter{ALC@unique}{0}
		\REQUIRE $w^{(-1)} = w^{(0)} = v^{(-1)} \in L^2(X)$ and $\gamma^{(0)} > 0$,  $\rho \in (0, 1)$, $\tau \in [1,2]$, $\gamma^{\max} >0$, $\kappa \in (0, \sqrt{2} -1)$
		\STATE $k \coloneqq 0$
		\STATE $\theta^{(0)} \coloneqq 1$
		\REPEAT
		\STATE $k \coloneqq k + 1$
		\STATE $i \coloneqq -1$
		\REPEAT\label{alg:parameters_malitsky_loop}
		\STATE $ i \coloneqq i+1$
		\STATE $\bar \theta \coloneqq
		\begin{cases}
			\sqrt{\frac{1+\tau \, \theta^{(k-1)}}{2\tau -1}}\rho^i,
				&
				\gamma^{(k-1)} \le \frac{1}{2}\gamma^{\max}
				\\
				\rho^i
				&
				\text{else}
		\end{cases}$
		\STATE $v^{(k)} \coloneqq w^{(k)} + \bar \theta \, (w^{(k)} - w^{(k-1)})$
		\label{alg:malitsky_prox_gradient_extrapolation}
		\STATE $\bar \gamma \coloneqq \paren[auto](){1 -\frac{1}{\tau}} \bar \theta \, \gamma^{(k-1)}$
		\UNTIL{$\bar \gamma \, \norm{A(v^{(k)}) - A(v^{(k-1)})}_{L^2(X)} \le \kappa \paren[auto](){1 -\frac{1}{\tau}} \norm{v^{(k)} - v^{(k-1)}}_{L^2(X)}$}
		\label{alg:parameters_malitsky_stopping}
		\STATE $\theta^{(k)} \coloneqq \bar \theta$
		\STATE $\gamma^{(k)} \coloneqq \bar \gamma$
		\STATE $w^{(k+1)} \coloneqq \proj[big]{\Delta_C \cap \bounds}{w^{(k)} - \gamma^{(k)} A(v^{(k)})}$
		\label{alg:malitsky_prox_gradient_prox}
		\UNTIL{a stopping criterion is fulfilled for $w^{(k+1)}$}
	\end{algorithmic}
\end{algorithm}

As initial values we choose $\kappa = 0.41$, $\rho =  0.7$, $\tau = 2$ as well as $\gamma^{\text{max}} = 10^5$.
Furthermore we set $w^{(-1)}$, $w^{(0)}$ and $v^{(-1)}$ constant and feasible.
For the initialization of $\gamma^{(0)}$ we adapt \cite[Remark~2.1]{Malitsky2017} which suggests to choose $w^{(-2)}$ and choose the largest $\gamma^{(0)}$ which fulfills 
\begin{equation*}
	\gamma^{(0)} \norm{a(w^{(-2)})- a(w^{(0)})} \le \kappa \norm{w^{(-2)} - w^{(0)}}
	.
\end{equation*}
Considering the discretized setting (cf.~\cref{sec:Discretization_DG0}) we construct $w^{(-2)}$ by first choosing a coordinate~$k$ such that $k \in \Argmax \setDef[auto]{1 \le i \le m}{\paren[auto](){A(w^{(0)})}_i}$.
Based on this we define
	\begin{equation*}
		\hat w_i \coloneqq  
		\begin{cases}
			w^{(0)}_i
			&
			\text{if } i \neq k,
			\\
			w^{(0)}_i + 2
			&
			\text{if } i = k,
		\end{cases}
	\end{equation*}
	and $w^{(-2)} = \proj{\Delta_C \cap \bounds}{\hat w}$. 
	By this approach we ensure that $w^{(-2)}$ is feasible, $w^{(-2)}_k = 1$ holds and the remaining entries of $w^{(-2)}$ are equal to the positive part of the corresponding entries of $w^{(0)}$ shifted.
	So by construction of $k$ as an index containing a largest entry of the gradient $A(w^{(0)})$, the vector $w^{(-2)}$ can be seen as an iterate prior to $w^{(0)}$.

%------------------------------------------------------------------
\subsection{Stopping Criterion}
\label{subsection:stopping_criterion}
%------------------------------------------------------------------

In this section we discuss a stopping criterion for both algorithms, which is based on a relaxation of the appropriate discrete version of \eqref{eq:w_pointwise}.
Define, for arbitrary $w \in \bounds$ (see \eqref{eq:bounded_in_L2_discrete}), the index sets
\begin{align*}
	J_0(w)
	&
	\coloneqq \setDef{i \in 1, \dots, m}{w_i = 0},
	\\
	J_{01}(w)
	&
	\coloneqq \setDef{i \in 1, \dots, m}{0 < w_i < 1},
	\\
	J_1 (w)
	&
	\coloneqq \setDef{i \in 1, \dots, m}{w_i =1}
\end{align*}
and set $z(w) \coloneqq - \Lambda^* \nabla F_q(\Lambda w)$.

The discrete counterpart of the necessary and sufficient optimality condition \eqref{eq:w_pointwise} for problem \eqref{eq:decomposition_into_F_G_H_primal}--\eqref{eq:definitions_of_G_H_primal} with $\alpha \ge 0$ reads
\begin{equation}
	\label{eq:w_coefficientwise}
	w \in \bounds
	\quad \text{and} \quad
	\paren[auto]\{\}{
		\begin{alignedat}{3}
			&
			z_i(w) 
			&&
			\le \zeta 
			&&
			\quad
			\text{on } J_0(w) 
			\\
			&
			z_i(w) - \alpha\,w_i
			&& 
			= \zeta 
			&&
			\quad
			\text{on } J_{01}(w) 
			\\
			&
			z_i(w) - \alpha 
			&&
			\ge \zeta
			&&
			\quad
			\text{on } J_1(w)
		\end{alignedat}
	}
	.
\end{equation}
Our stopping criterion is based on the relaxation of the optimality condition presented in \cref{lemma:relaxed_optimality_conditions}.
The discrete counterpart of \eqref{eq:w_pointwise_relaxed_estimated} reads
\begin{equation}
	\label{eq:w_coefficientwise_relaxed}
	w \in \bounds
	\quad \text{and} \quad
	\paren[auto]\{\}{
		\begin{alignedat}{3}
			& 
			z_i(w) 
			&
			\le  \zeta + \varepsilon 
			&
			&&
			\quad
			\text{on } J_0(w) 
			\\
			\zeta - \varepsilon \le {}
			&
			z_i(w) - \alpha\, w_i 
			&
			\le  \zeta + \varepsilon 
			&
			&&
			\quad
			\text{on } J_{01}(w) 
			\\
			\zeta - \varepsilon \le {} 
			&
			z_i(w) -\alpha
			&
			&
			&&
			\quad
			\text{on } J_1(w)
		\end{alignedat}
	}
	.
\end{equation}
A graphical illustration of \eqref{eq:w_coefficientwise_relaxed} is given in \cref{fig:example_shape}. 
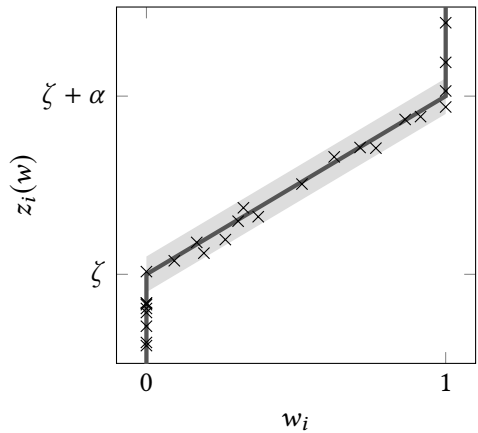
\begin{figure}[htb]
	\centering
	\setlength\figureheight{0.3\textwidth}
	\setlength\figurewidth{0.3\textwidth}
\begin{tikzpicture}

\pgftransformshift{\pgfpointanchor{current page}{center}}

%\fill [white, opacity=0.9] (-0.65\figurewidth,-0.6 \figureheight) rectangle (0.8\figurewidth,0.8\figureheight);

\begin{axis}[%
width=\figurewidth,
height=\figureheight,
at={(-0.4\figurewidth,-0.4\figureheight)},
scale only axis,
xmin=-0.1,
xmax=1.1,
ymin=1.5,
ymax=3.5,
xlabel={$w_i$},
ylabel={$z_i(w)$},
axis background/.style={fill=white},
xtick={0,1},
ytick={2,3},
yticklabels={$\zeta$,$\zeta + \alpha$},
legend style={legend cell align=left, align=left, draw=white!15!black}
]

\addplot [color=black, draw=none, mark=x, mark size=3pt, mark options={solid, black}]
  table[row sep=crcr]{%
0 1.41538933627946\\
0 1.00262208183163\\
0 0.641646185457853\\
0 1.08429926276052\\
0 1.48475932470098\\
0 1.70858306493708\\
0 0.663197514018374\\
0 0.653734580582475\\
0 0.724451429321634\\
0 1.39924027549845\\
0 1.8082504227046\\
0 1.5997364339364\\
0 1.83925891095452\\
0 1.82959454449872\\
0 0.36285984641615\\
0 1.78547036268116\\
0 2.01513559021656\\
0 1.61743467719034\\
0 1.8375322363391\\
0 0.141568300373483\\
};
\addlegendentry{data1}

\addplot [color=black, draw=none, mark=x,mark size=3pt, mark options={solid, black}]
  table[row sep=crcr]{%
0.324096590428357 2.37317378440877\\
0.373637675523219 2.32316685376076\\
0.766737557593821 2.70803163641561\\
0.168086658651728 2.17911782611942\\
0.519724683593696 2.50645483104007\\
0.627455609342921 2.65899498569675\\
0.713913685799429 2.71164592718263\\
0.306396105344489 2.29800864644248\\
0.263687109426303 2.19482848886853\\
0.916003507324704 2.88498495102687\\
0.0931750600170528 2.0760492545122\\
0.192028022891077 2.11877080214012\\
0.864500505417078 2.8690835284135\\
};
\addlegendentry{data2}

\addplot [color=black, draw=none, mark=x, mark size=3pt,mark options={solid, black}]
  table[row sep=crcr]{%
1 4.0359891708455\\
1 4.7004732041994\\
1 3.02889331240912\\
1 4.28825537675772\\
1 2.9390655480558\\
1 3.51131616678571\\
1 4.94503003586374\\
1 4.50574495396654\\
1 3.41173625343962\\
1 4.33244372820397\\
1 3.18949447147285\\
1 4.22260450075365\\
1 4.69973085458569\\
1 4.78957649986776\\
1 3.63157364468692\\
1 3.92125140590262\\
1 4.32698673390269\\
1 4.37865597152638\\
};

\addlegendentry{data3}

\addplot[area legend, draw=none, fill=gray1]
table[row sep=crcr] {%
x	y\\
0 1.9 \\
0 2.1 \\
1 3.1 \\
1 2.9 \\
}--cycle;
\addlegendentry{data4}

\addplot [color=gray2,solid, style = {ultra thick}]
table[row sep=crcr]{
0 1\\
0 2\\
1 3\\
1 4\\
};

\legend{};
\end{axis}
\end{tikzpicture}%
	\caption{Illustration of the stopping criterion \eqref{eq:w_coefficientwise_relaxed}. All pairs $(w_i,z_i(w))$ must lie in the light gray area of width~$2 \varepsilon$ or on one of the vertical lines, $i = 1, \ldots, m$.}
	\label{fig:example_shape}
\end{figure}

We now wish to derive of version of \eqref{eq:w_coefficientwise_relaxed} which does not require the evaluation of the shift value~$\zeta$.
To simplify the notation we introduce
\begin{alignat*}{2}
	\supent{0} 
	&
	\coloneqq \sup_{i \in J_0(w)} \paren[auto]\{\}{z_i(w)},
	&
	\infent{01}
	&
	\coloneqq \inf_{i \in J_{01}(w)} \paren[auto]\{\}{z_i(w)- \alpha\, w_i},
	\\
	\infent{1}
	&
	\coloneqq \inf_{i \in J_1(w)} \paren[auto]\{\}{z_i(w)- \alpha},
	\qquad
	&
	\supent{01}
	&
	\coloneqq \sup_{i \in J_{01}(w)} \paren[auto]\{\}{z_i(w)- \alpha\, w_i}.
\end{alignat*}
Notice that we write $\inf$ and $\sup$ instead of $\min$ and $\max$ since some of the index sets may be empty.
Based on these quantities, \eqref{eq:w_coefficientwise_relaxed} can be equivalently written as
\begin{subequations}
	\label{eq:coefficientwise_relaxed_simplyfied}
	\begin{alignat}{2}
		w \in \bounds
		\label{eq:coefficientwise_relaxed_simplyfied_1}
		\\
		\supent{0}-\zeta 
		&
		\le 
		\varepsilon,
		\qquad
		&
		\zeta - \infent{01} 
		&
		\le 
		\varepsilon,
		\label{eq:coefficientwise_relaxed_simplyfied_2}
		\\
		\supent{01} - \zeta 
		&
		\le 
		\varepsilon,
		\qquad
		&
		\zeta - \infent{1} 
		&
		\le 
		\varepsilon.
		\label{eq:coefficientwise_relaxed_simplyfied_4}
	\end{alignat}
\end{subequations}
Pairwise summation shows that \eqref{eq:coefficientwise_relaxed_simplyfied} implies 
\begin{subequations}
	\label{eq:coefficientwise_relaxed_removed_zeta}
	\begin{alignat}{2}
		w \in \bounds
		\label{eq:coefficientwise_relaxed_removed_zeta_1}
		\\
		\supent{0}-\infent{01} 
		&
		\le 
		2\varepsilon,
		\qquad
		&
		\supent{0} - \infent{1} 
		&
		\le 
		2\varepsilon
		\label{eq:coefficientwise_relaxed_removed_zeta_2}
		\\
		\supent{01} - \infent{01} 
		&
		\le 
		2\varepsilon,
		\qquad
		&
		\supent{01} - \infent{1} 
		&
		\le 
		2\varepsilon
		\label{eq:coefficientwise_relaxed_removed_zeta_3}
	\end{alignat}
\end{subequations}
Conversely, if \eqref{eq:coefficientwise_relaxed_removed_zeta} holds, there exists $\zeta \in \R$ such that
\begin{equation*}
	\max\paren[auto]\{\}{\supent{01},\supent{0}} -\varepsilon \le \zeta \le \min\paren[auto]\{\}{\infent{01},\infent{0}} + \varepsilon.
\end{equation*}
It is easy to see that this $\zeta$ satisfies \eqref{eq:coefficientwise_relaxed_simplyfied_2}--\eqref{eq:coefficientwise_relaxed_simplyfied_4}.
For example, the first inequality in \eqref{eq:coefficientwise_relaxed_simplyfied_2} holds since
\begin{equation*}
	\supent{0} - \zeta \le \supent{0} - \max\paren[auto]\{\}{\supent{01},\supent{0}} +\varepsilon \le \supent{0} - \supent{0} + \varepsilon = \varepsilon.
\end{equation*}
Thus we have the equivalence of \eqref{eq:w_coefficientwise_relaxed} and \eqref{eq:coefficientwise_relaxed_removed_zeta}.

Since both \cref{alg:fista,alg:malitsky_prox_gradient} maintain $w \in \Delta_C \cap \bounds$ by construction, we stop the iterations as soon as \eqref{eq:coefficientwise_relaxed_removed_zeta_2} and \eqref{eq:coefficientwise_relaxed_removed_zeta_3} holds.
To this end, we define the error function
\makeatletter
\ltx@ifclassloaded{svjour3}{%
\begin{multline}
	e(w) 
	\coloneqq 
	\\
	\frac{1}{2} \max\paren[auto]\{\}{\supent{0}-\infent{01}, \; \supent{0}-\infent{1}, \; \supent{01}-\infent{01}, \; \supent{01}-\infent{1}}
	.
	\label{eq:error}
\end{multline}
}{%
\begin{equation}
	e(w) 
	\coloneqq 
	\frac{1}{2} \max\paren[auto]\{\}{\supent{0}-\infent{01}, \; \supent{0}-\infent{1}, \; \supent{01}-\infent{01}, \; \supent{01}-\infent{1}}
	.
	\label{eq:error}
\end{equation}
}
\makeatother
Notice that $e(w)$ is finite since both $J_0(w) \cup J_{01}(w)$ and $J_{01}(w) \cup J_1(w)$ are non-empty.
This is due to \eqref{eq:assumption_C} and the feasibility of $w$.
Thus we obtain that \eqref{eq:coefficientwise_relaxed_removed_zeta} holds if any only if 
\begin{subequations}
	\label{eq:stopping_criterion}
	\begin{equation}
		e(w) \le \varepsilon
		.
	\end{equation}
	In our experiments, we choose the relative tolerance
	\begin{equation}
		\varepsilon \coloneqq 10^{-10} \cdot \paren[big](){\max_{i=1,\ldots, n} \paren[auto]\{\}{z_i(w)} - \min_{i=1,\ldots, n} \paren[auto]\{\}{z_i(w)}}.
	\end{equation}
\end{subequations}
As a safety measure, the algorithms are also terminated in case a maximum number of iterations is reached, which is described separately for each numerical example below.

%------------------------------------------------------------------
\subsection{Initial Comparison of Algorithms}
\label{subsec:comparison_algorithms}
%------------------------------------------------------------------

The aim of this section is to compare \cref{alg:fista} (FISTA) and \cref{alg:malitsky_prox_gradient} (PGMA) by solving an instance of problem \eqref{eq:decomposition_into_F_G_H_primal}--\eqref{eq:definitions_of_G_H_primal} both for the regularized case ($\alpha > 0)$ as well as the unregularized case ($\alpha = 0$).
We implemented \cref{alg:fista,alg:malitsky_prox_gradient} in \matlab~R2020a and ran them on an Ubuntu~18.04 machine with an Intel(R) Core(TM) i7-6700~CPU and 16~GiB of RAM.

We utilize the following example.
\begin{example}[Lotka-Volterra predator-prey model]
	\label{ex:predator_prey}
	We consider the nonlinear system
	\begin{equation}
		\label{eq:Lotka-Volterra}
		\begin{pmatrix}
			\dot{y_1}(t) 
			\\
			\dot{y_2}(t)
		\end{pmatrix}
		= 
		\begin{pmatrix*}[r]
			p_1 \, y_1(t) - p_3 \, y_1(t) \, y_2(t) 
			\\
			-p_2 \, y_2(t) + p_4 \, y_1(t) \, y_2(t) 
		\end{pmatrix*}
		,
	\end{equation}
	where $(y_1,y_2)$ denotes the number of prey/predators.
	The overall aim is to identify the parameter vector $p = (p_1, \dots, p_4)^\transp > 0$ by observations of the state component $y_1(t)$ for certain times $t \in [0,100]$.
	The selection of observation times is one of the design decisions to be made.
	Furthermore we also can vary the initial condition $y(0) = (y_{10},y_{20})^\transp \in [0,10]^2 $.
	This results in the design space $X = [0,10] \times [0,10] \times [0,100] \subset \R^3$.
	We discretize $X$ into $m = 30 \times 30 \times 30 = \num{27000}$ equal-sized cuboids.

	We consider various of the design criteria \eqref{eq:criteria} and in each case, identify optimal experimental conditions by approximately solving \eqref{eq:OED_regularized}--\eqref{eq:bounded_in_L2}.
	To set up an elementary Fisher information matrix (FIM) $\Upsilon(y_{10},y_{20},t)$, we proceed in the following way.
	We solve \eqref{eq:Lotka-Volterra} starting with initial conditions $y(0) = (y_{10},y_{20})^\transp$.
	The sensitivity derivatives $z(\cdot) \in \R^{2 \times 4}$ of the trajectory $y(\cdot)$ w.r.t.\ the parameter vector $p$ is given by the linear system
	\begin{equation}
		\label{eq:linearized_Lotka-Volterra}
		\begin{pmatrix}
			\dot{z_1}(t) 
			\\
			\dot{z_2}(t)
		\end{pmatrix}
		= 
		\begin{bmatrix*}[r]
			p_1 - p_3 \, y_2(t) & - p_3 \, y_1(t)
			\\
			p_4 \, y_2(t) & -p_2 + p_4 \, y_1(t)
		\end{bmatrix*}
		\begin{pmatrix}
			z_1(t) \\ z_2(t)
		\end{pmatrix}
		+
		\begin{bmatrix}
			y_1 & 0 & -y_1 y_2 & 0 \\
			0 & -y_2 & 0 & y_1 y_2
		\end{bmatrix}
		,
	\end{equation}
	endowed with initial conditions $z(0) = 0 \in \R^{2 \times 4}$.
	This follows easily from the implicit function theorem.
	The elementary FIM is then given by
	\begin{equation*}
		\Upsilon(y_{10},y_{20},t)
		=
		z_1(t)^\transp z_1(t)
		\in \R^{4 \times 4}
		.
	\end{equation*}
	Recall that we take $\Upsilon_i$ on each cell of the discretized design space $X$ to be its value in the midpoint.
	In order to precompute all \num{27000}~FIMs, we therefore need to solve 900 initial value problems \eqref{eq:Lotka-Volterra} and \eqref{eq:linearized_Lotka-Volterra} and evaluate each trajectory $z_1(\cdot)$ in 30~points in time.
	In practice, we solve the nonlinear forward problem \eqref{eq:Lotka-Volterra} and the sensitivity problem \eqref{eq:linearized_Lotka-Volterra} simultaneously for any given choice of initial conditions using an explicit Euler time stepping scheme with step size $\Delta t = 0.1$.
	We utilize the nominal parameter value $p = (0.1, 0.4, 0.02, 0.02)^\transp$.

	Each cell~$E_i$ in the design space has a volume of $\abs{E_i} = \frac{\num{10000}}{\num{27000}} = \frac{10}{27}$. 
	Finally we set $C = 5\cdot 10^{-4} \abs{X} = 5$, so that the experimental budget allows us to allocate a total weight of $0.05\%$ of all admissible experiments.
	This is equivalent to having 13.5~cells out of \num{27000} with weight~$w_i = 1$. 
\end{example}

In the experiment in this section, we choose $q = 0$ which amounts to the logarithmic D-criterion.
An optimal weight vector~$w$ in case $\alpha = 0$ can be seen in \cref{fig:run1_optimal_w}.
Here we have to point out that neither its support nor the solution itself is necessarily unique.
We mention that an a~posteriori sparsification approach similar as in the proof of \cite[Lemma~3.10]{NeitzelPieperVexlerWalter2018} could be used.
However, \cref{theorem:optimality_conditions}~\cref{item:optimality_conditions_4} does not apply since it is valid only in the continuous setting.

\begin{figure}[htb]
	\centering
	\includegraphics[width=0.8\textwidth]{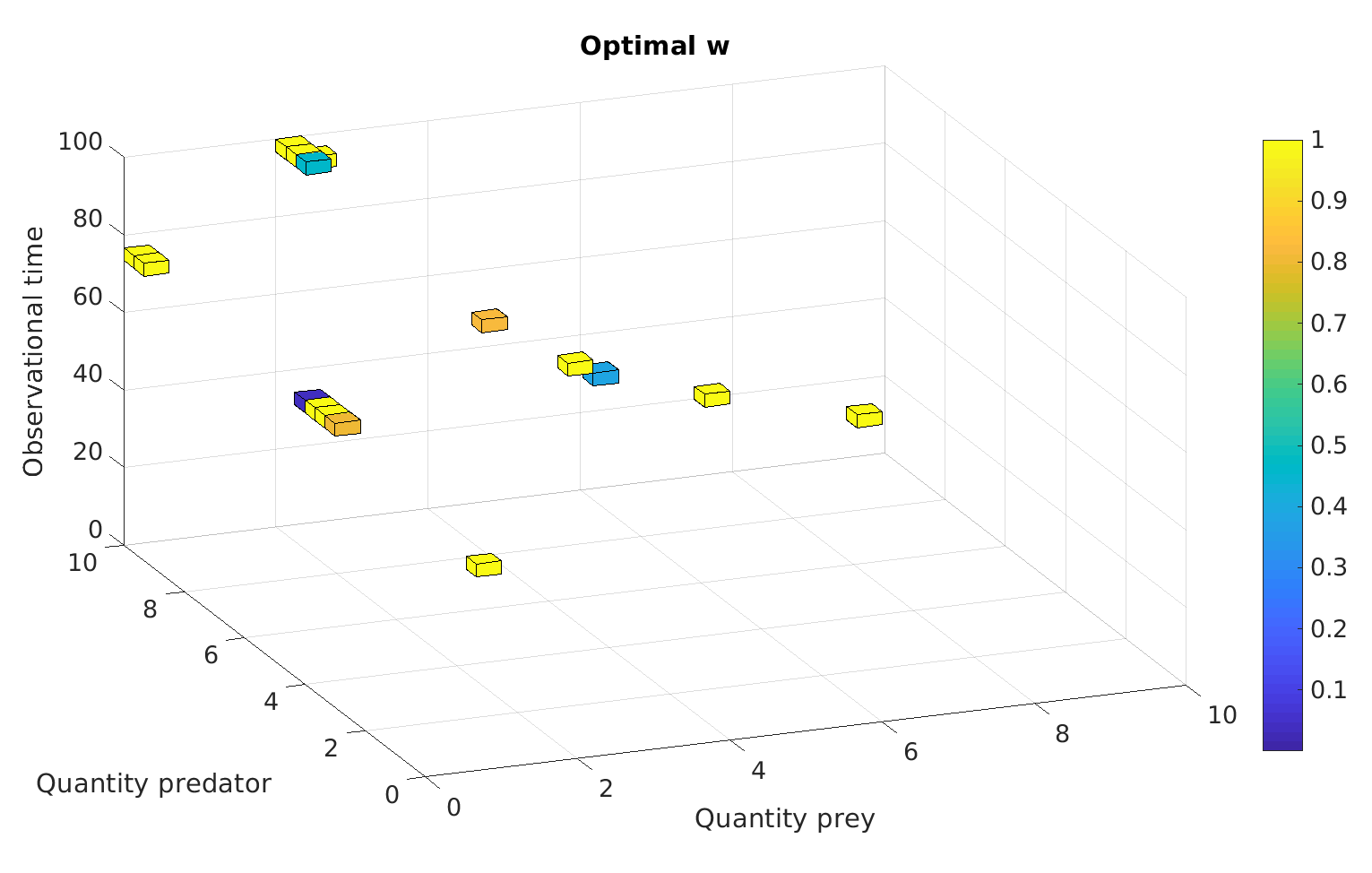}
	\caption{Solution of \cref{ex:predator_prey} for the logarithmic D-criterion ($q = 0$ in \eqref{eq:criteria}) with $\alpha = 0$. Entries equal to zero are not shown.}
	\label{fig:run1_optimal_w}
\end{figure}

Due to the fact that the cost per iteration for FISTA and PGMA is different, essentially due to different step size selection mechanisms, we utilize the elapsed CPU time as our performance criterion.
In each case, the setup time for the precomputation of the FIMs $\Upsilon_i$ is identical and it is not included.

\begin{figure}[htb]
	\centering
	\subfloat[]{%
		\includegraphics[width=0.45\textwidth]{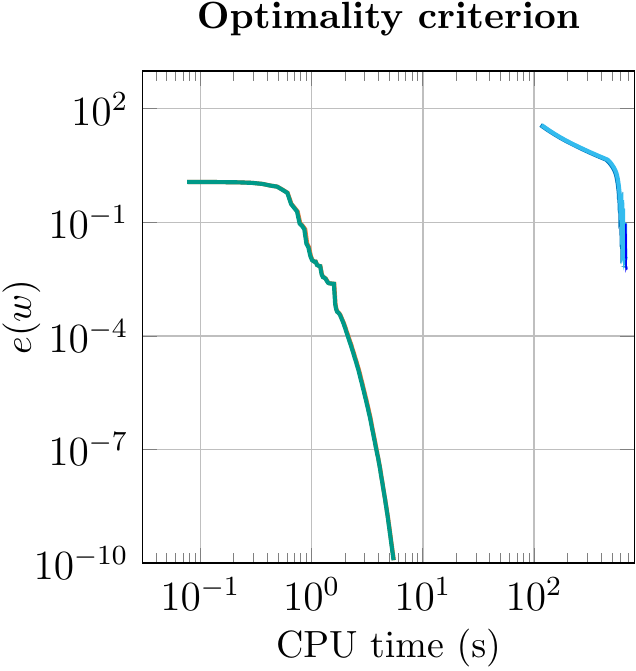}
		\label{fig:run1_pointwise_error}
		}%
	\hfill%
	\subfloat[]{%
		\includegraphics[width=0.45\textwidth]{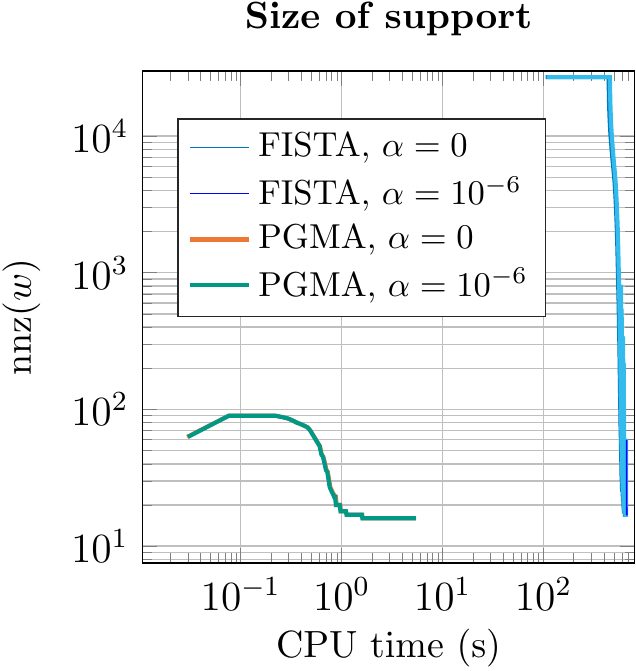}
		\label{fig:run1_support}
		}%
		\caption{\cref{ex:predator_prey}, comparison of algorithms by (a) value of $e(w)$ from \eqref{eq:error} and (b) size of the support of $w$.}
	\label{fig:run1_overview}
\end{figure}

\cref{fig:run1_pointwise_error} compares both algorithms for $\alpha = 0$ and $\alpha = 10^{-6}$ with respect to the decay of the error function \eqref{eq:error} over CPU time.
We clearly observe the PGMA outperforms FISTA by far.
The latter was stopped after \num{5000}~iterations without coming close to the desired tolerance.
By contrast, PGMA reached the stopping criterion \eqref{eq:stopping_criterion} after about 5.6~CPU seconds and within 300~iterations.

A further aspect of comparison of these algorithms is presented in \cref{fig:run1_support}, where the size of the support of the iterates is shown.
Again PGMA outperforms FISTA, although the latter reaches an iterate with almost the same degree of sparsity as the optimal solution after almost \num{5000}~iterations.
A partial explanation for the superiority of PGMA is based on the fact that it utilizes larger step sizes, in particular with regard to $\gamma_k$; see \cref{fig:run1_stepsizes}.
For FISTA, we observe that $\gamma_k$ is quite small and constant. 
Notice that FISTA does not allow $\gamma_k$ to increase, and its size is determined by the initial guess.
Furthermore $\theta_k$ is chosen a-priori, thus practically there is no adaptivity in the choice of the step sizes in FISTA.
PGMA starts from the same initial guess and thus also exhibits small values of $\gamma_k$ initially.
In constrast to FISTA, however, $\gamma_k$ is allowed to increase and does so until a reasonable magnitude is reached.
It is then about five orders of magnitude larger than for FISTA.
Notice also that $A$ is only locally Lipschitz, which suffices to give theoretical convergence guarantees for PGMA but not for FISTA.

\begin{figure}[htb]
	\centering
	\includegraphics[width=\linewidth]{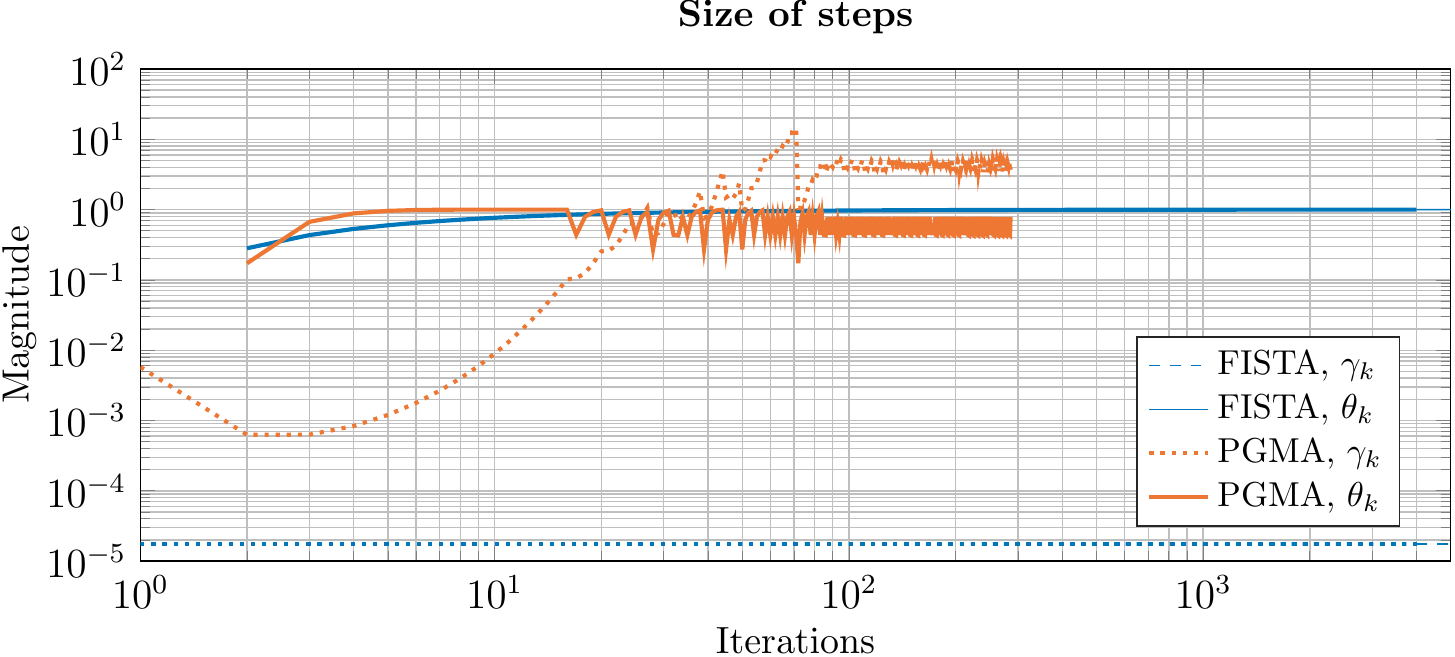}
	\caption{\cref{ex:predator_prey} with $\alpha = 0$, comparison of algorithms by step sizes $\gamma_k$ and $\theta_k$.}
	\label{fig:run1_stepsizes}
\end{figure}

Returning to \cref{fig:run1_pointwise_error}, we observe that both algorithms are only affected in a minor way by the choice of the regularization parameter~$\alpha$.
Apparently, the non-uniqueness of the optimal weight in the case $\alpha = 0$ does not represent an obstacle to the convergence.

%------------------------------------------------------------------
\subsection{Comparison of Different Design Criteria}
\label{subsec:comparison_design_criteria}
%------------------------------------------------------------------

In this section we compare the convergence behavior of PGMA for solving \cref{ex:predator_prey} with various design criteria, i.e., for different choices of $q$, as well as for various values of the regularization parameter~$\alpha$.
We chose $\alpha = 0$ and $\alpha = 10^{-3}$ and five different quantities of $q$ in order to observe different sparsity patterns.
All used combinations as well as the resulting size of the support of the solution is presented in \cref{tab:comparison_of_support_predator_prey}.
Recall that $q = 0$ refers to the logarithmic D-criterion while $q = 1$ denotes the A-criterion.
Furthermore we approximate the E-criterion by setting $q = 10$.

Again the comparison of both \cref{alg:fista,alg:malitsky_prox_gradient} is based on CPU time, excluding the setup time for the elementary FIMs~$\Upsilon_i$.

\begin{table}[htpb]
	\centering
	\caption{Comparison of the size of the support for \cref{fig:example_shape} with different design criteria.}
	\label{tab:comparison_of_support_predator_prey}
	\begin{tabular}{llllll}
		\toprule
		$\alpha$  & $q = 0$ &  $q=\frac{1}{2}$ &    $q = 1$ & $q=2$ &    $q = 10$ \\
		\midrule
		$0$       &      16 &          16 & 15 & 16 &          16 \\
		$10^{-3}$ &      16 & \num{27000} & \num{27000}& \num{27000}& \num{27000} \\
		\bottomrule
	\end{tabular}
\end{table}

The convergence results are presented in \cref{fig:run2_overview}.

\begin{figure}[htb]
	\centering
	\subfloat[]{%
		\includegraphics[width=0.45\textwidth]{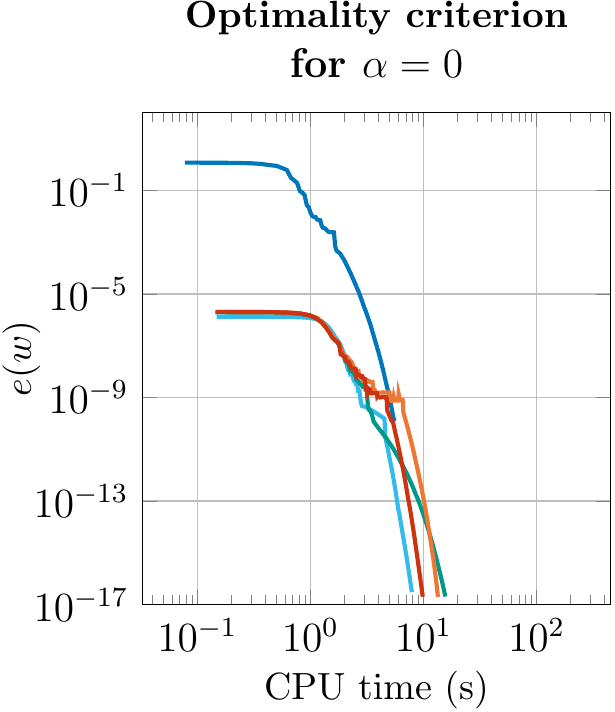}
		\label{fig:run2_error_alpha_0}
		}%
	\hfill%
	\subfloat[]{%
		\includegraphics[width=0.45\textwidth]{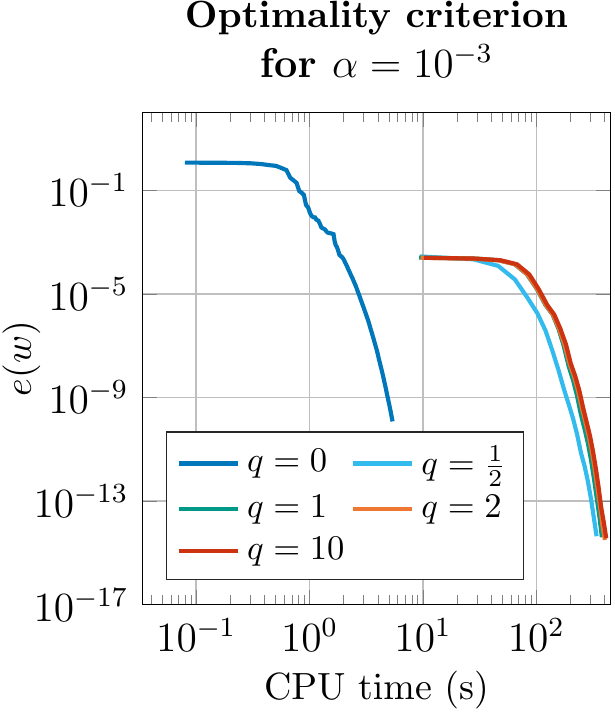}
		\label{fig:run2_error_alpha_greater_0}
		}%
		\caption{\cref{ex:predator_prey} solved with PGMA. Comparison of design criteria for different choices of $\alpha$ by the value of $e(w)$ from \eqref{eq:error}.}
	\label{fig:run2_overview}
\end{figure}

First we observe that in all variants PGMA, reached an iterate which fulfilled the stopping criterion \eqref{eq:stopping_criterion}.
Furthermore \cref{fig:run2_error_alpha_0,fig:run2_error_alpha_greater_0} indicate that in case $q = 0$ there are only minor differences in computational time for different values of~$\alpha$. 
Indeed PGMA took $5.8$ respectively $5.1$~seconds to converge for $\alpha = 0$ and $\alpha = 10^{-3}$.
In case $q > 0$ the computational time for $\alpha = 0$ and $\alpha = 10^{-3}$ is  about 10 and 400~seconds respectively, almost independently of the particular choice of $q$.

%------------------------------------------------------------------
\section{Acceleration Strategies}
\label{sec:simplicial_decomposition}
%------------------------------------------------------------------

In this section we present four acceleration strategies for the aforementioned algorithms.
All of them replace the solution of the discrete problem \eqref{eq:discretized_problem} by a sequence of smaller problems of the same or similar type.
Consequently, FISTA and PGMA can both serve as inner solvers but based on the findings of the previous section, we focus on PGMA.

The motivation for further acceleration is based on the interest of solving discretized OED problems for large values of~$m$ which arise from high dimensional design spaces and/or fine discretization.
We exploit the polyhedral structure of the feasible set in \eqref{eq:discretized_problem} and utilize that an optimal solution can be represented by a certain linear combination of vertices of $\Delta_C \cap \bounds$ or $\Delta_C$.

%------------------------------------------------------------------
\subsection{Simplicial Decomposition}
\label{subsec:SD}
%------------------------------------------------------------------

The first strategy we consider is the well known Simplicial Decomposion (\SD) as presented in \cite{Patriksson2009,VonHohenbalken1977}.
This approach utilizes that each point in the polyhedron $\Delta_C \cap \bounds$ can be written as a convex combination of vertices of $\Delta_C \cap \bounds$.
We denote these vertices by $s_i$, $i = 1, \ldots, r$.
Consequently, any $w \in \Delta_C \cap \bounds$ can be expressed as
\begin{equation}
	\label{eq:representation_SD}
	w = \sum_{i=1}^r \lambda_i \, s_i \text{ for some } \lambda_i \ge 0 \text{ satisfying } \sum_{i=1}^r\lambda_i = 1.
\end{equation}
Note that by straightforward geometric reasoning we obtain 
\begin{equation*}
	r \ge \max \paren[auto]\{\}{\nnz(w)-\neo(w),1}.
\end{equation*}
Recall that $\nnz$ and $\neo$ denote the numbers of non-zero entries, and of entries equal to one, respectively; see \eqref{eq:nnz_neo_and_others}.

In a nutshell, the simplicial decomposition algorithm restricts the search for an optimal weight vector in $\Delta_C \cap \bounds$ to the convex hull of a few active vertices of $\Delta_C \cap \bounds$.
In each iteration, a new vertex is added to the active set, and unused ones are removed.
The vertex added is one which yields the minimal value of the directional derivative of the objective.

The complete simplicial decomposition algorithm is described in \cref{alg:SD}. 

\begin{algorithm}[ht]
	\caption{Simplicial Decomposition (\SD)}
	\label{alg:SD}
	\begin{algorithmic}[1]
		\setcounter{ALC@unique}{0}
		\REQUIRE $\paren[auto]\{\}{s^{(0),i}}_{i=1}^{r^{(0)}}$ (vertices of $\Delta_C \cap \bounds \subset \R^m$) such that $\displaystyle \Lambda\paren[Bigg](){\sum_{i=1}^{r^{(0)}} s^{(0),i}} \succ 0$
		\STATE Set $R^{(0)} \coloneqq \paren[auto][]{s^{(0),1} \cdots\, s^{(0),r^{(0)}}}$ 
		\STATE Set $w^{(0)} \coloneqq \paren[auto](){r^{(0)}}^{-1} R^{(0)} \mathbf{1}$ and $k \coloneqq 1$ 
		\COMMENT{$w^{(0)}$ is the average of the $s^{(0),i}$}
		\REPEAT
		\STATE Compute a vertex $s^{(k)} \in \R^m$ which solves the linear program
		\begin{equation}
			\text{Minimize } s^\transp \Lambda^* \nabla F_q(\Lambda w^{(k)}) \quad \text{s.t. } s \in \Delta_C \cap \bounds
			\tag{\OUTSD}
			\label{eq:SP_SD}
		\end{equation}
		\STATE 
		Set $\bar R \coloneqq [R^{(k)}, \; s^{(k)}]$ 
		\label{item:SD_appending_vertex}
		\STATE Compute $\lambda^{(k+1)}$ as an inexact solution of the so-called restricted master problem
		\label{item:SD_RMP}
		\begin{equation}
			\text{Minimize } F_q(\Lambda \bar R \lambda) \quad \text{s.t. } \lambda \in \Delta_1
			\tag{\RMPSD}
			\label{eq:RMP_SD}
		\end{equation}
		To this end, we use \cref{alg:Torsney} with tolerance~$\delta$ in \eqref{eq:stopping_criterion_Torsney}.
		\STATE 
		\label{item:SD_rounding1}
		Set $\lambda^{(k+1)}_i \coloneqq \begin{cases} \lambda^{(k+1)}_i & \text{if } \lambda^{(k+1)}_i > \sqrt{\delta}/2 \\ 0 & \text{else }  \end{cases}$
		\STATE 
		\label{item:SD_rounding2}
		Set $\lambda^{(k+1)} = \norm{\lambda^{(k+1)}}_1^{-1} \lambda^{(k+1)}$
		\STATE
		Set $R^{(k+1)} \coloneqq \bar R[\,:\,,(\lambda^{(k+1)} > 0)] $
		\label{item:SD_purging}
		\COMMENT{purge unused vertices}
		\STATE 
		Set $w^{(k+1)} \coloneqq \bar R \lambda^{(k+1)}$
		\IF{the inital guess of $\lambda$ in \eqref{eq:RMP_SD} fulfills \eqref{eq:stopping_criterion_Torsney} in \cref{alg:Torsney} with tolerance~$\delta$}
		\STATE Set $\delta \coloneqq \delta / 10$
		\ENDIF
		\STATE 
		Set $k \coloneqq k+1$
		\UNTIL{stopping criterion \eqref{eq:stopping_criterion} is fulfilled}
		\label{alg:SD_stopping_criterion}
	\end{algorithmic}
\end{algorithm}

It is important to point out that we restrict our tests to cases $q \in [0,1]$ and $\alpha =0$, thus in particular only problems without regularization.
This apporach is based on the fact that we do not provide algorithms for the solution of the inner problem \eqref{eq:RMP_SD} in more general cases.

The initial values $\paren[auto]\{\}{s^{(0),i}}_{i=1}^{r^{(0)}}$ for \cref{alg:SD} are determined in the following way.
We first solve \eqref{eq:SP_SD} with $\Lambda^* \nabla F_q(\Lambda w)$ where~$w$ is a feasible vector with all entries equal.
This yields in $s^{(0),1}$.
If $\Lambda s^{(0),1}\succ 0$ does not hold, we iteratively insert further vertices of $\Delta_C \cap \bounds$ randomly until this condition is fulfilled for the sum of all considered vertices.\\

Next we discuss how the subproblems \eqref{eq:SP_SD} and \eqref{eq:RMP_SD} can be solved.

The linear program \eqref{eq:SP_SD} can be solved explicitly by sorting the entries of $\Lambda^* \nabla F_q(\Lambda w^{(k)})$ and assigning the entries equal to one and potentially entries in $(0,1)$ in $s^{(k)}$ in a greedy fashion.
In case several entries of $\Lambda^* \nabla F_q(\Lambda w^{(k)})$ agree, we make sure that the corresponding entries of $s^{(k)}$ agree as well.

Concerning the restricted master problem \eqref{eq:RMP_SD}, notice that this problem varies in dimension from iteration to iteration, depending on the number of active vertices, i.e., columns of $\bar R$.
For the subsequent discussion we omit the iteration index and denote this dimension by~$r$.
In order to compare \eqref{eq:RMP_SD} with the settings as in \cref{sec:Discretization_DG0}, we introduce the FIM associated with the $j$-th vertex in the set of active vertices $\bar R$ as
\begin{equation}
	\bar \Upsilon_j \coloneqq \sum_{i= 1}^m \Upsilon_i \abs{E_i} \bar R_{i,j} = \Lambda \bar R[\,:\,,j]
	.
\end{equation}
It is worth noting that $\bar \Upsilon_j$ is symmetric and positive semi-definite since $\Upsilon_i$ has this property and the entries in $\bar R$, consisting of vertex coordinates of $\Delta_C \cap \bounds$, are non-negative.
Next we represent the synthesis operator \eqref{eq:synthesis_operator_discrete} in terms of the barycentric coordinates $\lambda \in \R^r$:
\begin{equation}
	\bar \Lambda \lambda 
	\coloneqq 
	\Lambda \bar R \lambda 
	= 
	\sum_{j=1}^r \lambda_j \, \bar \Upsilon_j 
	.
	\label{eq:definition_bar_Lamba_SD}
\end{equation}
The weight constraint is expressed as
\begin{equation}
	C = \mathbf{1}^\transp M w = \mathbf{1}^\transp M \bar R \lambda = \sum_{i = 1}^m \abs{E_i} \sum_{j = 1}^r \bar R_{i,j} \lambda_j = \sum_{j = 1}^r  \lambda_j  \sum_{i = 1}^m \abs{E_i} \bar R_{i,j} = C \sum_{j = 1}^r  \lambda_j 
\end{equation}
and thus it reduces to $\mathbf{1}^\transp \lambda = 1$.
To summarize, \eqref{eq:RMP_SD} can be written as 
\begin{equation}
	\label{eq:reformulated_RMP_SD}
	\begin{aligned}
		\text{Minimize} \quad & F_q(\bar \Lambda \lambda) , \quad \lambda \in \R^r \\
		\text{s.t.} \quad & \lambda \ge 0 \\
		\text{and} \quad & \mathbf{1}^\transp \lambda = 1
		.
	\end{aligned}
\end{equation}
Notice that, in contrast to \eqref{eq:discretized_problem}, \eqref{eq:reformulated_RMP_SD} does not feature pointwise upper bounds on the variable $\lambda$.

Unfortunately, the first-order methods discussed in \cref{sec:numerical_results} are not efficiently applicable for the solution of the resctriced problems \eqref{eq:reformulated_RMP_SD}.
The reason is that both \cref{alg:fista} (FISTA) and \cref{alg:malitsky_prox_gradient} (PGMA) require the orthogonal projection onto $\Delta_C \cap \bounds$.
In \eqref{eq:reformulated_RMP_SD}, the appropriate inner product is given by  
\begin{equation}
	\paren[auto](){\bar R \lambda_1, \bar R \lambda_2}_M =  \lambda_1^\transp \bar R^\transp M \bar R \lambda_2 = \lambda_1 ^\transp \bar M \lambda_2
\end{equation}
with $\bar M \coloneqq \bar R^\transp M \bar R$.
In contrast to $M$, the reduced inner product matrix $\bar M$ is, in general, not diagonal.
Therefore, the evaluation of the projection cannot be achieved as in \cref{alg:projection_feasible_set} but it becomes significantly more expensive.
In addition, $\bar M$ is, in general, only positive semi-definite.

In order to present numerical results for \cref{alg:SD}, we resort to a simple solver for the inner problem \eqref{eq:reformulated_RMP_SD}.
As mentioned in the introduction, an approach described in the literature is Torsney's multiplicative algorithm; see \cref{alg:Torsney}.
It was described and analyzed in \cite{SilveyTitteringtonTorsney1978:1,Torsney2009:1,Yu2010:1} and specifically in \cite{UcinskiPatan2007,HerzogRiedelUcinski2017:1} in the context of simplicial decomposition.
We point out that \cref{alg:Torsney} is only applicable when $\alpha = 0$, i.e., for the unregularized problem, and for $q \in [0,1]$ (cf.\ \cite{Yu2010:1}).

\begin{algorithm}[ht]
	\caption{Torsney's algorithm for \eqref{eq:RMP_SD} ($\alpha = 0$ and $q \in [0,1]$)}
	\label{alg:Torsney}
	\begin{algorithmic}[1]
			\setcounter{ALC@unique}{0}
		\REQUIRE $\lambda^{[0]} \in \R^r$
		\STATE 
			Set $j \coloneqq 0$
		\REPEAT
			\STATE 
				Set 
				\begin{equation*}
					\lambda^{[j+1]}_i 
					\coloneqq
					\lambda^{[j]}_i \frac{\nabla F_q (\bar \Lambda \lambda^{[j]}) \dprod \bar \Upsilon_i}{\nabla F_q (\bar \Lambda \lambda^{[j]}) \dprod \bar \Lambda \lambda^{[j]}}, \quad i = 1, \ldots, r
				\end{equation*}
			\STATE 
				Set $j \coloneqq j + 1$
				\UNTIL{$\lambda^{[j]}$  fulfills the stopping criterion \eqref{eq:stopping_criterion_Torsney}}
		\label{alg:Torsney_stopping_criterion}
	\end{algorithmic}
\end{algorithm}

\cref{alg:Torsney} is very easy to implement. 
Notice that we denote the iteration counter $\cdot^{[j]}$ with square brackets to avoid a confusion with the outer iteration index $\cdot^{(k)}$ in \cref{alg:SD}.
Due to the multiplicative nature of \cref{alg:Torsney} we have to initialize it with $\lambda^{[0]}$ strictly positive.
In the very first call to \cref{alg:Torsney}, we initialize $\lambda^{[0]}$ as a multiple of an all-ones vector.
In subsequent calls to \cref{alg:Torsney}, we utilize the final iterate of the previous call to create a more informed initial guess.
To be precise, we remove unused coordinates, initialize the new coordinate with $1/r$ and rescale the remaining entries so that the total sum equals one.

We stop \cref{alg:Torsney} as soon as 
\begin{equation}
	\label{eq:stopping_criterion_Torsney}
	\lambda_i^{[j]} \, \paren[big](){\max_{k = 1,\ldots,r} \nabla F_q (\bar \Lambda \lambda^{[j]})\dprod \bar \Upsilon_k - \nabla F_q (\bar \Lambda \lambda^{[j]})\dprod \bar \Upsilon_i}
	\le
	\delta
\end{equation}
holds (cf. also \cref{fig:stopping_criterion_torsney}) or a maximal number of iterations is reached.

\begin{figure}[htb]
	\centering
	\setlength\figureheight{0.3\textwidth}
	\setlength\figurewidth{0.3\textwidth}
	\input{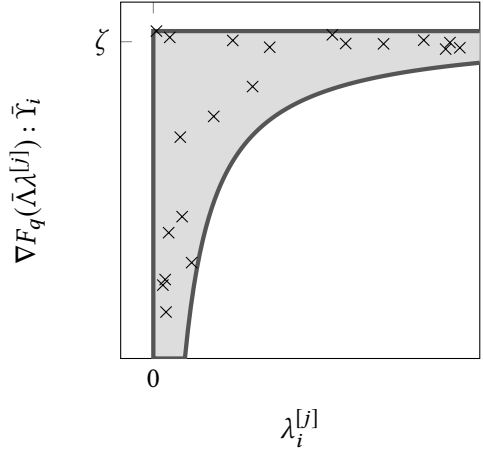}
	\caption{Illustration of the stopping criterion \eqref{eq:stopping_criterion_Torsney}. All pairs $\paren[auto](){\lambda_i^{[j]},\nabla F_q (\bar \Lambda \lambda^{[j]})\dprod \bar \Upsilon_i}$ must lie in the light gray area or on the bold lines, $i = 1, \ldots, r$.}
	\label{fig:stopping_criterion_torsney}
\end{figure}

We conclude the description of the simplicial decompositon \cref{alg:SD}, combined with \cref{alg:Torsney} as inner solver, by noting that it requires a rounding strategy; see \cref{item:SD_rounding1,item:SD_rounding2} in \cref{alg:SD}.
We first set all entries of $\lambda^{[j]}$ which are less then $\sqrt\delta/2$ to zero and later rescale the remaining ones such that all entries of the resulting vector sum up to~1.
Therefore we can only expect to obtain inexact solutions of \eqref{eq:discretized_problem}.
Furthermore we observe that due to the non-uniqueness of the solutions of \eqref{eq:RMP_SD}, purging vertices in \cref{alg:SD} does not necessarily reduce the subproblem to its minimal size.

In spite of these obstacles, we included the classical simplicial decomposition \cref{alg:SD} for comparison with more efficient accelerated solvers described in the following subsection.

\subsection{Methods Based on Active Set Strategies}
\label{subsec:GASS}

In this subsection we consider three further acceleration strategies which are based on a common framework given in \cref{alg:meta_active_set}.
We term this framework the generic active set strategy (GASS).
In comparison with the simplicial decomposition (SD)~\cref{alg:SD}, the inner problems to be solved in each iteration are different.
The three variants of GASS we consider are termed the simplicial decomposition modified (SDM), simplicial decomposition modified with heuristics (SDMH), and primal-dual semiactive set strategy (PDSAS).

To simplify the notation we define $\Deltafree_C \coloneqq \setDef{w \in \R^m}{\inner{1}{w}_M = C}$.
Furthermore we have $\bounds^- \coloneqq \setDef{w \in \R^m}{w \ge 0}$ and $\bounds^+ \coloneqq \setDef{w \in \R^m}{w \le 1}$ and thus $\bounds = \bounds^- \cap \bounds^+$.
The main difference is that this algorithm represents its iterates $w^{(k)}$ as bound constrained convex combinations of vertices of $\Delta_1 = \Deltafree_1 \cap \bounds^-$, rather than arbitrary convex combinations of vertices of $\Deltafree_C \cap \bounds$.
In other words, any $w$ in $\Deltafree_C \cap \bounds$ can be written as
\begin{equation}
\label{eq:representation_modified_SD}
w = \sum_{i=1}^r \lambda_i \, s_i \text{ for some } 0 \le \lambda_i \le 1 \text{ satisfying } \sum_{i=1}^r \abs{E_i} \, \lambda_i = C,
\end{equation}
where $s_i$, $i = 1, \ldots, r$ are vertices of $\Delta_1$.
The idea of applying this decomposition is based on previous work by \cite{Rutenberg1970} (cf.\ also \cite{Hsia1974}) and \cite{Marin1995}.
Notice that the representation \eqref{eq:representation_modified_SD} requires $r \ge \nnz(w)$, which is clearly larger than in \eqref{eq:representation_SD}, which was used in \cref{alg:SD}.

In order to better understand the combinatorial effort for solving \eqref{eq:discretized_problem}, we have to compare the number of vertices in $\Deltafree_C \cap \bounds$ and in $\Delta_1$.
The latter is clearly equal to~$m$.
We do not attempt to specify the exact number of vertices for $\Deltafree_C \cap \bounds$ here.
However, in the special case $\abs{E_i} \equiv \abs{E}$ for all $i = 1, \ldots, m$ and $C/\abs{E}$ is integer, then there are precisely 
\begin{equation*}
\binom{m}{C/\abs{E}}
\end{equation*}
vertices in $\Delta_C \cap \bounds$, which is generally much larger than~$m$.

The generic active set strategy in \cref{alg:meta_active_set} utilizes the representation \eqref{eq:representation_modified_SD}.
The general idea is to split the constraints describing the feasible set $\bounds$ w.r.t.\ the bound constraints into two parts. 
We thus obtain two sets $\bounds_{\text{inner}}$ and $\bounds_{\text{outer}}$ satisfying $\bounds = \bounds_{\text{inner}} \cap \bounds_{\text{outer}}$.
 
By replacing the representation \eqref{eq:representation_SD} with \eqref{eq:representation_modified_SD} we have to identify a larger number of active vertices out of a much smaller set of admissible vertices.
As a consequence, we expect \cref{alg:meta_active_set} to perform better in terms of the number of outer iterations, at the mild expense of having bound constraints in the inner problem.

In this subsection we say that a vertex of $\Delta_1$ is active if the corresponding entry of $w$ is active regarding the lower bounds, e.g. the entry is equal zero.
Note that this is different from the simplicial decomposition approach discussed in \cref{subsec:SD}, where active vertices corresponded to one ore more entries of $w$ being positive.

\begin{algorithm}[ht]
	\caption{Generic Active Set Strategy (GASS)}
	\label{alg:meta_active_set}
	\begin{algorithmic}[1]
		\setcounter{ALC@unique}{0}
		\REQUIRE $w^{(0)} \in \Delta_C \cap \bounds$ such that $\Lambda w^{(0)} \succ 0$
		\REQUIRE splitting of the feasible set w.r.t.\ the pointwise bound constraints $\bounds_{\text{inner}} \cap \bounds_{\text{outer}} = \bounds$
		\STATE Set $k = 0$
		\STATE Set $\AA^{(0)} = \setDef{1 \le i \le m}{w_i = 0}$ 
		\REPEAT
		\STATE Outer problem: 
		\begin{equation}
			\text{ Find the set of active vertices }\AA^{(k+1)} 
			\tag{\OUTGASS}
			\label{eq:out_meta_active_set}
		\end{equation}
		\STATE Inner problem: Compute $w^{(k+1)}$ as an approximate solution of
		\begin{equation}
			\begin{aligned}
				\text{Minimize} \quad & F_q\paren[auto](){\Lambda w} + \frac{\alpha}{2}\norm{w}_M^2,  \quad w \in \Deltafree_C \cap \bounds_{\text{inner}} \\ 
				\text{s.t.} \quad & w_i = 0 \quad \text{ for all } i \in \AA^{(k+1)}
			\end{aligned}
			\tag{\INGASS}
			\label{eq:in_meta_active_set}
		\end{equation}
		\STATE 
		Set $k \coloneqq k+1$
		\UNTIL{$w^{(k)} \in \bounds_{\text{outer}}$ and stopping criterion \eqref{eq:stopping_criterion} is fulfilled}
		\label{alg:meta_active_set_stopping_criterion}
	\end{algorithmic}
\end{algorithm}

The numerical effort of solving \cref{alg:meta_active_set} mainly depends on the splitting of the upper and lower bound constraint set $\bounds$ into $\bounds_{\text{inner}}$ and $\bounds_{\text{outer}}$.
Based on that, different strategies to determine $\AA^{(k+1)}$ in \eqref{eq:out_meta_active_set} as well as for solving \eqref{eq:in_meta_active_set} arise.
In \cref{tab:active_set_algorithms} we give an overview over the approaches we consider.

\begin{table}[htp]
\centering
\begin{tabular}{c c c c c}
	\toprule
	\multirow{2}{*}{
		Algorithm
	}
	&
	\multirow{2}{*}{
		$\bounds_{\text{outer}}$
	}
	&
	\multirow{2}{*}{
		$\bounds_{\text{inner}}$
	}
	&
	approach for
	&
	algorithm for 
	\\
	&
	&
	&
	\eqref{eq:out_meta_active_set}
	&
	\eqref{eq:in_meta_active_set}
	\\
	\midrule
	SDM
	&
	$\R^m$
	&
	$\bounds$
	&
	\cref{alg:inactive_entries_SDM}
	&
	\cref{alg:malitsky_prox_gradient}
	\\
	SDMH
	&
	$\R^m$
	&
	$\bounds$
	&
	\cref{alg:inactive_entries_SDMH}
	&
	\cref{alg:malitsky_prox_gradient}
	\\
	PDSAS
	&
	$\bounds^-$
	&
	$\bounds^+$
	&
	\cref{alg:inactive_entries_PDSAS}
	&
	\cref{alg:malitsky_prox_gradient}
	\\
	\bottomrule
\end{tabular}
\caption{Overview of active set algorithms (variants of \cref{alg:meta_active_set}).}
\label{tab:active_set_algorithms}
\end{table}

First we discuss the algorithms SDM and SDMH utilizing the strategy for the active sets described in \cref{alg:inactive_entries_SDM} as well as \cref{alg:inactive_entries_SDMH}.
Both algorithms are a straightforward modification of the classical SD~\cref{alg:SD}, considering that $w$ is a particular linear combination of vertices of $\Delta_1$.
SDM and SDMH differ in terms of the number of vertices inserted in each iteration, i.e., how much $\abs{\AA^{(k+1)}}$ can shrink compared to $\abs{\AA^{(k)}}$ in \eqref{eq:out_meta_active_set}.
While SDM frees exactly one vertex, SDMH frees all local minimizers of the gradient.
Therefore we expect that the size of the inner problems in SDM will be smaller, possibly at the expense of an increased number of outer iterations.

\begin{algorithm}[ht]
	\caption{Approach for \eqref{eq:out_meta_active_set} in SDM} 
	\label{alg:inactive_entries_SDM}
	\begin{algorithmic}[1]
		\setcounter{ALC@unique}{0}
		\REQUIRE previous active set $\AA^{(k)}$ and previous $w^{(k)}$
		\STATE Compute $s^{(k)} \in \paren[auto]\{\}{1, \dots, m}$ which fulfills
		\begin{equation*}
			s^{(k)} = \argmin \setDef[auto]{\paren[auto](){\Lambda^* \nabla F_q(\Lambda w^{(k)})}_i + \alpha\,w^{(k)}_i}{i \in \AA^{(k)}}
		\end{equation*}
		\STATE Set $\AA^{(k+1)} \coloneqq \setDef[auto]{i \in \paren[auto]\{\}{1, \dots, m}}{w_i^{(k)} = 0} \setminus \paren[auto]\{\}{s^{(k)}} $
	\end{algorithmic}
\end{algorithm}

\begin{algorithm}[ht]
	\caption{Approach for \eqref{eq:out_meta_active_set} in SDMH}
	\label{alg:inactive_entries_SDMH}
	\begin{algorithmic}[1]
		\setcounter{ALC@unique}{0}
		\REQUIRE previous active set $\AA^{(k)}$ and previous $w^{(k)}$
		\STATE Compute $S^{(k)} \subset \AA^{(k)}$ as the set of all indicies corresponding to local minimizers (with respect to $X$) of the mapping 
		\begin{equation*}
			\AA^{(k)} \ni i \mapsto \paren[auto](){\Lambda^* \nabla F_q(\Lambda w^{(k)})}_i + \alpha\,w^{(k)}_i \in \R.
		\end{equation*}
		\STATE Set $\AA^{(k+1)} \coloneqq \setDef[auto]{i \in \paren[auto]\{\}{1, \dots ,m}}{w_i^{(k)} = 0} \setminus S^{(k)}$
	\end{algorithmic}
\end{algorithm}

The solution of problem \eqref{eq:in_meta_active_set} by \cref{alg:malitsky_prox_gradient} benefits from a good initial guess.
	When $k = 0$, we use a constant vector on the active coordinates.
	In subsequent iterations of \cref{alg:meta_active_set}, we make use of the previous solution $w^{(k)}$ with unused coordinates removed and new entries initialized to~$0$.

The third algorithm we consider is a primal-dual semiactive set strategy (PDSAS).
Its name derives from the fact that a primal-dual active set approach similarly to \cite{BergouniouxItoKunisch1999,HintermuellerItoKunisch2002} is applied, but solely w.r.t.\ the lower bound constraints $\bounds^-$.
The reason to leave the upper bound constraints to \eqref{eq:in_meta_active_set} is that otherwise, the active constraints may become incompatible with the mass constraint $w \in \Deltafree_C$.

As the name indicates we will make use of the dual variables of \eqref{eq:in_meta_active_set}.
In order to motivate these dual variables we recall \eqref{eq:discretized_problem}:
\begin{equation*}
	\begin{aligned}
		\text{Minimize} \quad & F_q(\Lambda w) + \frac{\alpha}{2} \norm{w}_M^2, \quad w \in \R^m \\
		\text{s.t.} \quad & 0 \le w \le 1 \\
		\text{and} \quad & \mathbf{1}^\transp M w = C
		.
	\end{aligned}
\end{equation*}

Next we indroduce the dual variables $\mu^- \in \R^m$ for the lower and $\mu^+ \in \R^m$ for the upper bounds as well as $\zeta \in \R$ for the weight constraint.
The corresponding KKT conditions can be formulated as
\begin{align*}
	0 &= \Lambda^* \nabla F\paren[auto](){\Lambda w} + \alpha \, w  - \mu^- + \mu^+ + \zeta \mathbf{1}  \\
	0 &\le w \, \perp \, \mu^- \ge 0\\
	0 &\le 1-w \perp \, \mu^+ \ge 0
\end{align*}
since $M$ is spd and diagonal.
By applying the complementarity function $\R^2 \ni (a,b) \mapsto \max\paren[auto]\{\}{a,cb} \in \R $ for arbitrary $c > 0$ and replacing $\mu^-$ by $\nu \coloneqq -\mu^-$, this system is seen to be equivalent to
\begin{subequations}
	\begin{align}
		\mu^+ & = \max\paren[auto]\{\}{-\Lambda^* \nabla F\paren[auto](){\Lambda w} - \alpha \, w  - \zeta \mathbf{1}, \; 0} 
		\\
		\nu & = \min\paren[auto]\{\}{-\Lambda^* \nabla F\paren[auto](){\Lambda w} - \alpha \, w  - \zeta \mathbf{1}, \; 0} 
		\label{eq:reformulated_kkt_mu_minus}  
		\\
		0 & = \nu - \min\paren[auto]\{\}{0, c \, w + \nu} 
		\label{eq:reformulated_kkt_active_set} 
		\\
		0 & \le 1-w \, \perp \, \mu^+ \ge 0
		.
	\end{align}
\end{subequations}

In contrast to SDM and SDMH, we make use of the dual variables $\nu$ and $\zeta$ in order to estimate the active set $\AA \coloneqq \setDef{1 \le i \le m}{w_i = 0}$ in each iteration.
Algorithmically this can be done by computing $\nu^{(k)}$ via \eqref{eq:reformulated_kkt_mu_minus} with $w^{(k)}$ as the solution of the previous inner problem and $\zeta^{(k)}$ the corresponding multiplier of the weight constraint.
Next by \eqref{eq:reformulated_kkt_active_set} we clearly have 
\begin{equation*}
	c \, w^{(k)} + \nu < 0 \quad \Rightarrow \quad \nu > 0 \quad \Rightarrow \quad w = 0
\end{equation*}
what will be used to compute $\AA^{(k+1)}$.
Furthermore we also have to consider the case that for given $w^{(k)}$, $\nu^{(k)}$ as well as  $c$, the set $\AA^{(k+1)}$ is so large that
\begin{equation}
	C \le \sum_{\mrep{i \notin \AA^{(k+1)}}{}} \abs{E_i}
	\label{eq:infeasibility_sas}
\end{equation}
is violated.
In this case, the inner problem \eqref{eq:in_meta_active_set} is not feasible.
It turns out that by increasing $c$ in these cases we can ensure the feasibility of the inner problem. 

The described approach is summarized in \cref{alg:inactive_entries_PDSAS}.
As initial guess for \eqref{eq:in_meta_active_set} we use the constant feasible vector in each iteration of \cref{alg:meta_active_set}.

\begin{algorithm}[ht]
	\caption{Approach for \eqref{eq:out_meta_active_set} in PDSAS}
	\label{alg:inactive_entries_PDSAS}
	\begin{algorithmic}[1]
		\setcounter{ALC@unique}{0}
		\REQUIRE previous $w^{(k)}$, $\zeta^{(k)}$, parameter $c$
		\ENSURE $\AA^{(k+1)}$ and $c>0$
		\STATE Set $\nu^{(k)} \coloneqq \min\paren[auto]\{\}{-\Lambda^*\nabla F_q(\Lambda w^{(k)}) - \alpha \, w^{(k)} -\zeta^{(k)}, \; 0}$
		\REPEAT
		\STATE Set $\AA^{(k+1)} \coloneqq \setDef[big]{1 \le i \le m}{\nu^{(k)}_i + c \ w^{(k)}_i < 0}$
		\IF{\eqref{eq:infeasibility_sas} is violated}
		\STATE Set $c \coloneqq 10c$ 
		\ENDIF
		\UNTIL{\eqref{eq:infeasibility_sas} is fulfilled}
	\end{algorithmic}
\end{algorithm}

%------------------------------------------------------------------
\subsection{Comparison of Accelerated Algorithms 1}
\label{subsec:comparison_accelerated_algorithms_1}
%------------------------------------------------------------------

In this section we compare the accelerated algorithms presented in \cref{tab:active_set_algorithms} as well as PGMA.
Again we make use of \cref{ex:predator_prey}, considering $\alpha = 0$ as well as $q=0$, i.e., the logarithmic D-criterion.
We refine the discretization of $X$ to contain $50^3 = \num{125000}$ $\DG_0$-elements.
Besides the regular stopping criterion we stopped the algorithms from \cref{tab:active_set_algorithms} after $300$~iterations when necessary. 
For the solution of each restricted master problem, see \cref{item:SD_RMP} of \cref{alg:SD}, we allowed up to $5 \cdot 10^4$ iterations of Torsney's \cref{alg:Torsney}.
The same iteration limit is used for \cref{alg:malitsky_prox_gradient} for the solution of \eqref{eq:in_meta_active_set}.
The results are described in \cref{fig:SD_overview}.
As before, the time to setup the elementary Fisher information matrices is excluded from all timings.

\begin{figure}[htb]
\centering
\subfloat[]{%
	\includegraphics[width=0.42\textwidth]{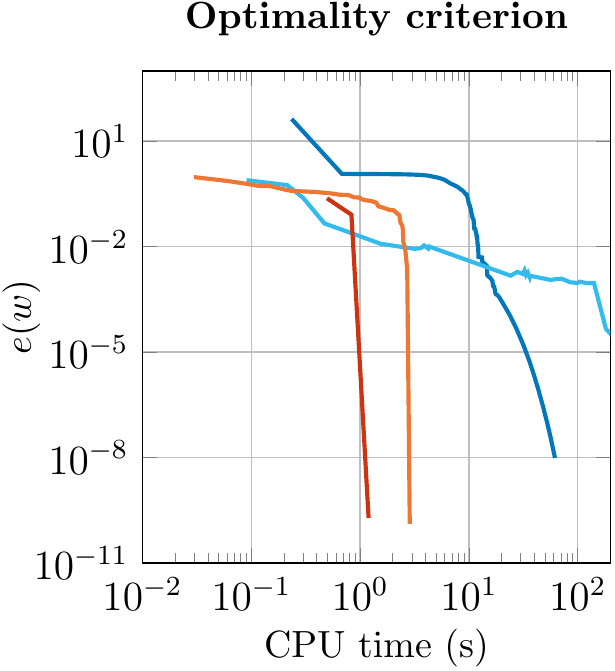}
	\label{fig:SD_pointwise_error}
	}%
\hfill%
\subfloat[]{%
	\includegraphics[width=0.42\textwidth]{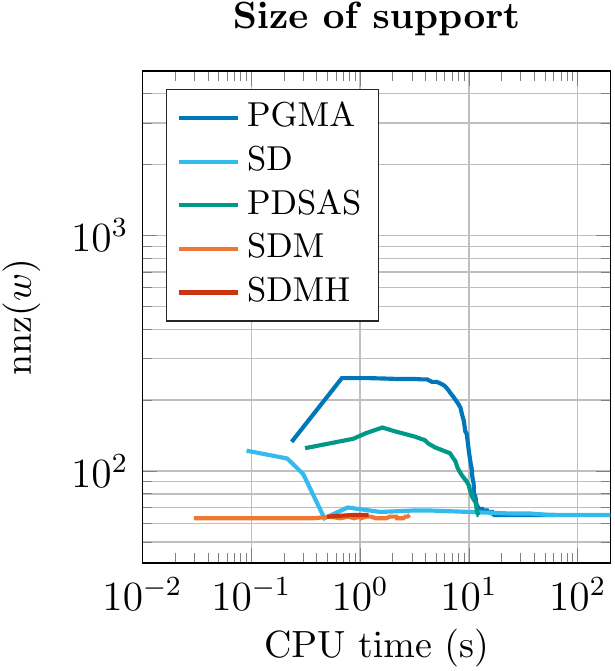}
	\label{fig:SD_support}
	}%
	\caption{\cref{ex:predator_prey} with $m = \num{125000}$; comparison of accelerated algorithms for $\alpha = 0$ by (a) value of $e(w)$ from \eqref{eq:error} and (b) size of the support.}
\label{fig:SD_overview}
\end{figure}

First we compare the violation of the contraints versus the CPU time elapsed.
The corresponding chart can be found in \cref{fig:SD_pointwise_error}.
All algorithms beside SD reached an iterate which fulfills the stopping criterion \eqref{eq:stopping_criterion}.
Note that all iterates of PDSAS exept the last one are infeasible.
(Due to numerical errors this is also the case of some iterates of SD.)
Therefore PDSAS is omitted from \cref{fig:SD_pointwise_error}.
While PDSAS needed 13.7~seconds, SDM took about 2.8~seconds and SDMH took less than 1.5~seconds to terminate.
For comparison, the unaccelerated PGMA (\cref{alg:malitsky_prox_gradient}) required 62~seconds to run.

Next in \cref{fig:SD_support} we compare the number of non-zero elements of the iterates.
Obviously the usage of PGMA results in iterates with the largest sizes of support.
The first iterates of SD also have a large support, what can be explained by the fact that each used vertex increases the size of the support possibly by more than one.
Considering PDSAS we plotted the quantity of positive entries of the iterates.
Nevertheless due to the construction of the algorithm every iterate execpt the final one contains negative entries.
Such a behavior is not observed by SDM or SDMH since here the inserted vertices have only one non-zero element each.
Thus the size of the support of the iterates does not grow larger than the support of the solution.
All algorithms reached an iterate with 65 non-zero elements.
Nevertheless one should keep in mind that the results of SD are rounded while the others are not.

Let us finally compare the behavior of the algorithms in more detail.
While it took less than 0.1~seconds to compute the first iterate in SD and SDM, SDMH required about 0.8~seconds.
But already the third iterate of SDMH (after 1.5~seconds) fulfilled the stopping criterion \eqref{eq:stopping_criterion}, whereas the other algorithms needed much more time to terminate.
This behavior confirms our expectations described in \cref{subsec:GASS}, that SDMH might require fewer outer iterations but generally has larger inner problems than SDM.

%------------------------------------------------------------------
\subsection{Comparison of Accelerated Algorithms 2}
\label{subsec:comparison_accelerated_algorithms_2}
%------------------------------------------------------------------

In this subsection, we consider only the two best algorithms from \cref{subsec:comparison_accelerated_algorithms_1}, SDM and SDMH.
The example is taken from \cite{NeitzelPieperVexlerWalter2018} and it is even larger than our previous examples.
\begin{example}[Stationary diffusion problem]
	\label{ex:stationary_diffusion}
	We consider the nonlinear model
	\begin{equation*}
		\begin{aligned}
			- \nabla \cdot \paren[big](){\exp(m_p)\nabla y} & = 0 & & \text{in } \Omega \coloneqq (0,1)^2 \\
			y & = x_1 & & \text{on } \Gamma_D \coloneqq \{0,1\} \times (0,1) \\
			\exp(m_p) \, \partial_n y & = 0 & & \text{on } \partial \Omega \setminus \Gamma_D,
		\end{aligned}
	\end{equation*}
	where $m_p(x) = \sum_{i=1}^5 \sum_{j=1}^5 p(i,j) \sin(\pi i x_1) \sin(\pi j x_2)$.

	Since this is a sensor placement problem, the design space~$X$ equals the domain~$\Omega$.
	The sensitivity of $y$ w.r.t $p(i,j)$ will be denoted by $\delta y_{i,j}$ and fulfills
	$-\nabla\cdot \paren[auto](){\exp(m_q)\cdot  \nabla \delta y_{i,j}} = \nabla \cdot \paren[auto](){\exp(m_q) \sin(i \pi x_1) \sin(j \pi x_2) \nabla y}$ in $\Omega$ together with the boundary conditions $\delta y_{i,j} = $ on $\Gamma_D$ and $\exp(m_q)\partial_n \delta y_{i,j} = 0$ on $\partial \Omega \setminus \Gamma_D$.
	Arranging the full set of sensitivities into the matrix
	\begin{equation*}
		\delta y =
		\begin{pmatrix}
			\delta y_{1,1}
			&
			\dots
			&
			\delta y_{1,5}
			\\
			\vdots
			&
			&
			\vdots
			\\
			\delta y_{5,1}
			&
			\dots
			&
			\delta y_{5,5}
		\end{pmatrix}
	\end{equation*}
	we find the following expression for the elementary Fisher information matrices:
	\begin{equation*}
		\Upsilon(x)
		=
		\vectorize(\delta y(x))^\transp \vectorize(\delta y(x))
		\in \R^{25 \times 25}
		.
	\end{equation*}
	Recall that we take $\Upsilon_i$ on each cell of the discretized design space $X$ to be its value in the midpoint.
	We choose $p = 0 \in \R^{5 \times 5} \simeq \R^{25}$ as nominal value of the parameter.
	We discretize $\Omega$ as well as $X$ by identical meshes containing $m \approx \num{430000}$ triangular elements.
	Each element has maximal edge length of $0.0057$ but the elements are not precisely of equal size.
	We set $C = 10^{-4} \abs{X} = 10^{-4}$, so the experimental budget allows us to allocate a total weight of $0.01\%$ of the sum of the weights of all admissible experiments.
\end{example}

We aim to solve \cref{ex:stationary_diffusion} for $\alpha = 0$ and $q = 0$, i.e., we consider the logarithmic D-criterion for the unregularized problem.
The preprocessing of the data $\Upsilon_i$ (which is once again not included in the timings) takes about $610$~CPU seconds.
Based on the observations in \cref{subsec:comparison_accelerated_algorithms_1} we only present the results for SDM and SDMH (see \cref{tab:active_set_algorithms}), since the remaining methods exhibit a much slower convergence behavior.
Similarly as in the previous example, besides the stopping criteria already decribed, the algorithm for solving the inner problems \eqref{eq:in_meta_active_set} also terminates when $10^4$ iterations are reached.

\begin{figure}[htb]
\centering
\subfloat[]{%
	\includegraphics[width=0.45\textwidth]{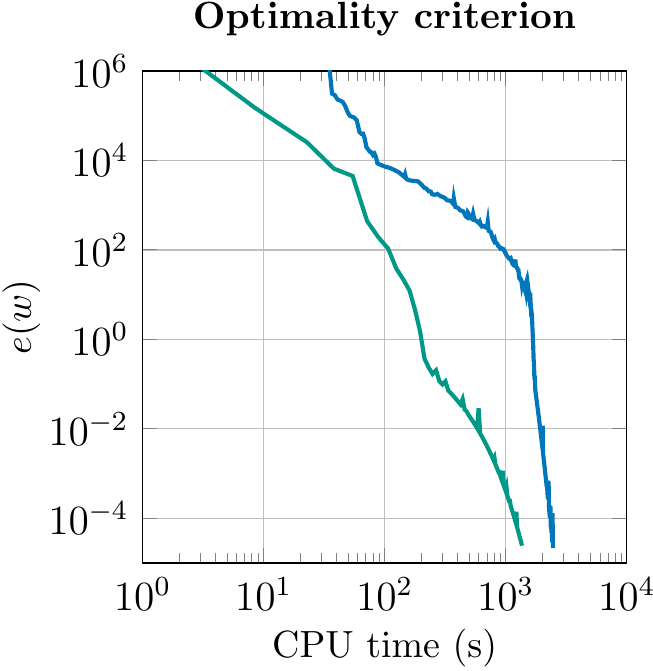}
	\label{fig:SD2_pointwise_error}
	}%
\hfill%
\subfloat[]{%
	\includegraphics[width=0.45\textwidth]{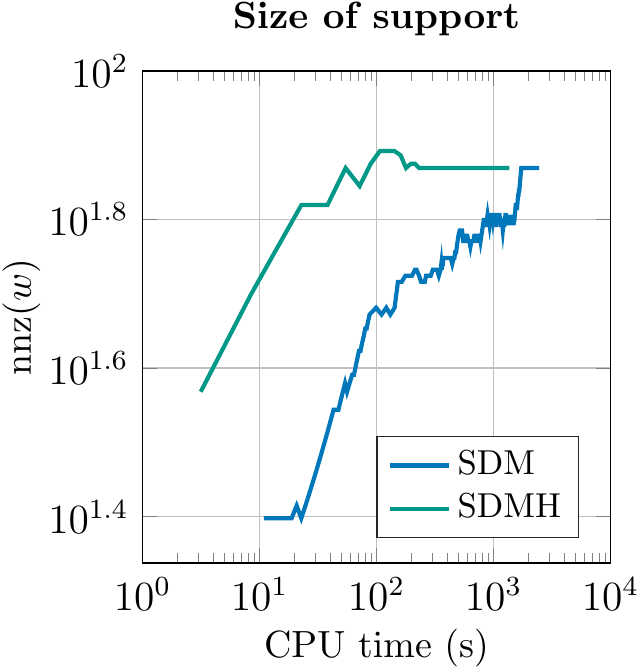}
	\label{fig:SD2_support}
	}%
	\caption{\cref{ex:stationary_diffusion} for $\alpha = 0$ and $q = 0$, Comparison of accelerated algorithms with PGMA as algorithm for the RMP's by (a) value of $e(w)$ from \eqref{eq:error} and (b) size of the support out of a maximum of $m \approx \num{430000}$.}
\label{fig:SD2_overview}
\end{figure}

\Cref{fig:SD2_pointwise_error} compares the satisfaction of the optimality criterion as measured by \eqref{eq:error}.
The algorithms SDM and SDMH needed $\num{2400}$ and $\num{1300}$~CPU seconds, respectively.
\Cref{fig:SD2_support} describes the support of the iterates over the time.
Similarily as in the previous example, SDM and SDMH typically underestimate the size of the support until convergence.

\section{Discussion}

In this paper we discussed first-order methods (FISTA and PGMA) for a class of optimal experimental design (OED) problems, which contain pointwise upper bounds as well as a maximal total weight.
PGMA exhibited a significanly more favorable convergence behavior, which can be explained by the fact that FISTA is unable to increase the stepsize~$\gamma_k$.
It turns out that this strategy is not ideal for the problem class under consideration.

Subsequently, we discussed four acceleration strategies.
These comprised the simplicial decomposition method (\cref{alg:SD}) with Torsney's \cref{alg:Torsney} as inner solver, as well as three variants (SDM, SDMH, PDSAS) of a generic active set strategy (GASS).
The latter are based on the idea of working with vertices of the simplex $\Delta_C$ rather than vertices of polyhedron $\Delta_C \cap \bounds$ .
As was illustrated in \cref{fig:SD_overview}, the methods of class GASS outperform PGMA as well as SD.

The algorithms of class GASS differ w.r.t.\ the strategy of how many entries of $w^{(k)}$ may become inactive in one iteration.
While SDM allows this only for one entry per iteration, SDMH allows several.
As expected, allowing more entries to become active may decrease the number of outer iterations at the cost of an increasing size of the inner problems.
For the examples considered throughout this paper, SDMH outperformed SDM, but we make no claim that this is always the case.
We also considered a version of the primal-dual active set strategy, which we termed PDSAS.
In our experiments, this strategy was not as effective SDM or SDMH but still outperformed simplicial decomposition with Torsney's method as inner solver.

The precomputation of the elementary Fisher information matrices~$\Upsilon_i$ is a significant part of the overall run time for large scale problems.
Therefore, there is opportunity to develop adaptive discretization strategies of the design space to further reduce computational cost.

\appendix

%------------------------------------------------------------------
\section{Proof of \cref{lemma:technical_results}}
\label{section:proof_lemma:technical_results}
%------------------------------------------------------------------

\begin{proof}
	\Cref{item:monotonicity_W_P}:
	Recall the definitions $W(v) \coloneqq \sum_{i=1}^m \abs{E_i} v_i$ and $P(\zeta) \coloneqq \proj{\bounds}{f-\zeta}$ from \eqref{eq:nnz_neo_and_others}, where $\bounds = [0,1]^m$.
	For arbitrary numbers $\zeta_1 > \zeta_2$ we get
	\begin{equation*}
		W(P(\zeta_1)) - W(P(\zeta_2)) = W( P(\zeta_1) - P(\zeta_2)) \le 0 
	\end{equation*}
	since $P(\zeta_1) \le P(\zeta_2)$ and $\abs{E_i} > 0$ holds for all $i = 1, \dots, m$.

	\Cref{item:monotonicity_zeta}:
	We split the proof into two cases. 
	Recall that the function $\zeta$ was defined in \eqref{eq:definition_of_zeta_function} and that we assumed that the entries in~$f$ are sorted in descending order.
	First we handle the case $\ell^* = k^*$, where obviously $\sum_{i=1}^{k^*} \abs{E_i} = C$ holds.
	By a simple reformulation we get
	\begin{align*}{2}
		\zeta(k^*,k^*-1)
			&=
			\abs{E_{k^*}}^{-1} \paren[auto](){\sum_{i=1}^{k^*-1}\abs{E_i} + \abs{E_{k^*}} f_{k^*} - C} 
			\\
			&=
			f_{k^*} + \abs{E_{k^*}}^{-1} \paren[auto](){\sum_{i=1}^{k^*-1}\abs{E_i}  - C}
			\\
			&=
			f_{k^*} + \abs{E_{k^*}}^{-1} (- \abs{E_{k^*}})
			\\
			&=
			f_{k^*} -1
			\\
			&=
			\zeta(k^*,k^*).
	\end{align*}
	In the second case, i.e., $0 \le \ell \le k^* -2$, we have from \eqref{eq:definition_of_zeta_function}
	\begin{align*}
		\paren[auto](){\sum_{i=\ell+1}^{k^*} \abs{E_i}} \, \zeta(k^*,\ell)
			&
			=
			\sum_{i=1}^\ell\abs{E_i} + \sum_{i=\ell+1}^{k^*} \abs{E_i} f_i - C
			\\
			&
			=
			\sum_{i=1}^{\ell+1}\abs{E_i} + \sum_{i=\ell+2}^{k^*} \abs{E_i} f_i - C -\abs{E_{\ell+1}} + \abs{E_{\ell+1} } f_{\ell+1}
			\\
			&
			=
			\paren[auto](){\sum_{i=\ell+2}^{k^*} \abs{E_i}} \, \zeta(k^*,\ell+1) + \abs{E_{\ell+1}} (f_{\ell+1}-1) 
			.
	\end{align*}
	By further reformulation we get
	\begin{alignat*}{2}
			&
			&
			\paren[auto](){\sum_{i=\ell+1}^{k^*} \abs{E_i}} \, \zeta(k^*,\ell)
			&
			=
			\paren[auto](){\sum_{i=\ell+2}^{k^*} \abs{E_i}} \, \zeta(k^*,\ell+1)  +\abs{E_{\ell+1}} (f_{\ell+1}-1) 
			\\
			\Leftrightarrow \quad
			&
			&
			\abs{E_{\ell+1}} \, \zeta(k^*,\ell) + \paren[auto](){\sum_{i=\ell+2}^{k^*} \abs{E_i}} \, \zeta(k^*,\ell) 
			&
			=
			\paren[auto](){\sum_{i=\ell+2}^{k^*} \abs{E_i}} \, \zeta(k^*,\ell+1)  +\abs{E_{\ell+1}} (f_{\ell+1}-1) 
			\\			
			\Leftrightarrow \quad
			&
			&
			\abs{E_{\ell+1}} (\zeta(k^*,\ell)-f_{\ell+1}+1)
			&
			=
			\paren[auto](){\sum_{i=\ell+2}^{k^*} \abs{E_i}} \, \paren[big](){\zeta(k^*,\ell+1) - \zeta(k^*,\ell)} 
			. 
	\end{alignat*}
	Since $\abs{E_i} > 0$ holds for all $i$, we obtain \eqref{eq:monotonicity_zeta}.

	\Cref{item:feasibility_lstar_equal_kstar}:
	Note that due to $k^* =\ell^*$ and assumption \eqref{eq:assumption_C} we have $k^* = \ell^* < m$.
	First we prove the statment for $i=k^*$:
	\begin{equation*}
		f_{k^*+1} - \zeta(k^*,k^*) = f_{k^*+1} -(f_{k^*} -1) < 1,
	\end{equation*}
	where we used the definition of $\zeta$ in \eqref{eq:definition_of_zeta_function} and the fact that  $f_{k^*+1}$ is stricly less than $f_{k^*}$ since the entries of $f$ are sorted and $v_{k^*+1} = \proj{[0,1]}{f_{k^*+1}-\zeta^*} = 0 < v_{k^*}$.

	Now we consider $i \in \{ 1,\dots, k^*-1 \}$.
	Note since $k^* = \ell^*$ we have $\sum_{i=1}^{k^*} \abs{E_i} = C$ and thus
	\begin{align*}
		f_{i+1} -\zeta(k^*,i) 
			&
			=
			f_{i+1} - \paren[auto](){\sum_{j=i+1}^{k^*} \abs{E_j}}^{-1} \paren[auto](){ \sum_{j=1}^i\abs{E_j} + \sum_{j=i+1}^{k^*} \abs{E_j} f_j - C }
			\\
			&
			=
			f_{i+1} - \paren[auto](){\sum_{j=i+1}^{k^*} \abs{E_j}}^{-1} \paren[auto](){ -\sum_{j=i+1}^{k^*}\abs{E_j} + \sum_{j=i+1}^{k^*} \abs{E_j} f_j  }
			\\
			&
			=
			f_{i+1} - \paren[auto](){\sum_{j=i+1}^{k^*} \abs{E_j}}^{-1} \paren[auto](){ \sum_{j=i+1}^{k^*} \abs{E_j} (f_j-1)  }
			\\
			&\ge
			f_{i+1} - \paren[auto](){\sum_{j=i+1}^{k^*} \abs{E_j}}^{-1} \paren[auto](){ \sum_{j=i+1}^{k^*} \abs{E_j}   } (f_{i+1}-1)
			\\
			&
			=
			1
			.
	\end{align*}
	This estimate is based on the fact that $f_j \le f_{i+1}$ holds for all $j \ge i+1$.

	\Cref{item:feasibility_lstar_smaller_kstar}:
	First we observe that if $f_{j+1} - \zeta(k^*,j) < 1$ holds for some $j \in \{ \ell_{\min}, \dots, k^*-1 \}$, then this is also the case for all $i \in \{ j+1, \dots, k^*-1 \}$.
	This can be proved iteratively since for $j \le k^*-2$ one gets
	\begin{equation}
		1 > f_{j+1} - \zeta(k^*,j) \ge f_{j+2} - \zeta(k^*,j) > f_{j+2} - \zeta(k^*,j+1)
		\label{eq:feasibility_monotonicity}
	\end{equation}
	by utilizing \eqref{eq:monotonicity_zeta}.

	Now the following strategy is applied iteratively for increasing $i \in \{ \ell_{\min}, \dots, k^*-1 \}$.
	If $1 \le f_{i+1} - \zeta(k^*,i)$ holds, we can increase $i$ by one.
	Otherwise $1 > f_{i+1} - \zeta(k^*,i)$ holds, which is fulfilled at least for $i=r$.
	For $i < k^*-1$ we can apply \eqref{eq:feasibility_monotonicity}, which proves the claim.
\end{proof}

% Insert bibliography
\printbibliography

\end{document}